\documentclass{amsart}
\usepackage[all]{xy}
\usepackage{verbatim}
\usepackage{color}
\usepackage{amsthm}
\usepackage{amssymb}
\usepackage[colorlinks=true]{hyperref}
\usepackage{footmisc}
\usepackage{stmaryrd}

\numberwithin{equation}{section}

\newtheorem{theorem}[equation]{Theorem}
\newtheorem*{theorem*}{Theorem} \newtheorem{lemma}[equation]{Lemma}

\newtheorem*{conjecture*}{Mamma Conjecture}
\newtheorem*{conjecture1*}{Mamma Conjecture (revisited)}
\newtheorem{proposition}[equation]{Proposition}
\newtheorem{corollary}[equation]{Corollary}
\newtheorem*{corollary*}{Corollary}

\theoremstyle{remark}

\newtheorem{example}[equation]{Example}
\newtheorem{subexample}[equation]{Subexample}

\newtheorem{notation}[equation]{Notation}

\theoremstyle{remark}
\newtheorem{remark}[equation]{Remark}

\newcommand{\cA}{{\mathcal A}}
\newcommand{\cB}{{\mathcal B}}
\newcommand{\cC}{{\mathcal C}}
\newcommand{\cD}{{\mathcal D}}

\newcommand{\cF}{{\mathcal F}}
\newcommand{\cG}{{\mathcal G}}
\newcommand{\cH}{{\mathcal H}}

\newcommand{\cL}{{\mathcal L}}

\newcommand{\cN}{{\mathcal N}}
\newcommand{\cO}{{\mathcal O}}

\newcommand{\cQ}{{\mathcal Q}}

\newcommand{\cW}{{\mathcal W}}
\newcommand{\cX}{{\mathcal X}}
\newcommand{\cY}{{\mathcal Y}}
\newcommand{\cZ}{{\mathcal Z}}

\newcommand{\bbA}{\mathbb{A}}
\newcommand{\bbB}{\mathbb{B}}
\newcommand{\bbC}{\mathbb{C}}

\newcommand{\bbF}{\mathbb{F}}

\newcommand{\bbP}{\mathbb{P}}

\newcommand{\bbQ}{\mathbb{Q}}
\newcommand{\bbZ}{\mathbb{Z}}

\DeclareMathOperator{\id}{id}

\DeclareMathOperator{\NChow}{NChow} 
\DeclareMathOperator{\NNum}{NNum} 

\newcommand{\dgcat}{\mathrm{dgcat}} 

\newcommand{\perf}{\mathrm{perf}}

\newcommand{\dg}{\mathrm{dg}}

\newcommand{\uHom}{\underline{\mathrm{Hom}}}
\newcommand{\Hom}{\mathrm{Hom}}

\newcommand{\rep}{\mathrm{rep}}

\newcommand{\Hmo}{\mathrm{Hmo}}
\newcommand{\op}{\mathrm{op}}

\newcommand{\too}{\longrightarrow}

\newcommand{\ie}{\textsl{i.e.}\ }

\let\oldmarginpar\marginpar
\def\marginpar#1{\oldmarginpar{\tiny #1}}

\begin{document}

\title[HPD-invariance of the Tate, Beilinson and Parshin conjectures]{HPD-invariance of the \\Tate, Beilinson and Parshin conjectures}
\author{Gon{\c c}alo~Tabuada}
\address{Gon{\c c}alo Tabuada, Department of Mathematics, MIT, Cambridge, MA 02139, USA}
\email{tabuada@math.mit.edu}
\urladdr{http://math.mit.edu/~tabuada}
\thanks{The author was supported by a NSF CAREER Award}

%
\date{\today}
\abstract{We prove that the Tate, Beilinson and Parshin conjectures are invariant under Homological Projective Duality (=HPD). As an application, we obtain a proof of these celebrated conjectures (as well as of the strong form of the Tate conjecture) in the new cases of linear sections of determinantal varieties and complete intersections of quadrics. Furthermore, we extend the original conjectures of Tate, Beilinson and Parshin  from schemes to stacks and prove these extended conjectures for certain low-dimensional global orbifolds.}}

\maketitle


\section{Introduction and statement of results}
Let $k:=\bbF_q$ be a finite field of characteristic $p$ with $q=p^n$, $W(k)$ the associated ring of $p$-typical Witt vectors, and $K:=W(k)[1/p]$ the fraction field of $W(k)$. Given a smooth projective $k$-scheme $X$, we will write $\cZ^\ast(X)_\bbQ$ for the (graded) $\bbQ$-vector space of algebraic cycles on $X$ up to rational equivalence, $\cZ^\ast(X)_\bbQ/_{\!\sim \mathrm{num}}$ for the quotient of $\cZ^\ast(X)_\bbQ$ with respect to the numerical equivalence relation, $H^\ast_{\mathrm{crys}}(X):=H^\ast_{\mathrm{crys}}(X/W(k))\otimes_{W(k)}K$ for the crystalline cohomology groups of $X$, and $K_\ast(X)_\bbQ$ for the $\bbQ$-linearized algebraic $K$-theory groups of $X$.

Given a prime number $l\neq p$, consider the associated cycle class map
\begin{equation}\label{eq:class-map}
\cZ^\ast(X)_{\bbQ_l} \too H^{2\ast}_{l\text{-}\mathrm{adic}}(X_{\overline{k}}, \bbQ_l(\ast))^{\mathrm{Gal}(\overline{k}/k)}
\end{equation}
with values in $l$-adic cohomology. In the sixties, Tate \cite{Tate} conjectured the following:

\vspace{0.2cm}

Conjecture $T^l(X)$: {\it The cycle class map \eqref{eq:class-map} is surjective}.

\vspace{0.2cm}

In the same vein, consider the cycle class map 
\begin{equation}\label{eq:cycle-crystalline}
\cZ^\ast(X)_{\bbQ_p} \too H^{2\ast}_{\mathrm{crys}}(X)(\ast)^{\mathrm{Fr}_p}
\end{equation}
with values in the $\bbQ_p$-vector subspace of those elements which are fixed by the crystalline Frobenius $\mathrm{Fr}_p$. Following \cite{Milne-AIM}, Tate's conjecture admits the $p$-version:

\vspace{0.2cm}

Conjecture $T^p(X)$: {\it The cycle class map \eqref{eq:cycle-crystalline} is surjective}.

\vspace{0.2cm}

In the eighties, Beilinson (see \cite[Conj.~50]{Handbook}) conjectured the following:

\vspace{0.2cm}

Conjecture $B(X)$: {\it The equality $\cZ^\ast(X)_\bbQ= \cZ^\ast(X)_\bbQ/_{\!\sim \mathrm{num}}$ holds}.

\vspace{0.2cm}

Also in the eighties, Parshin (see \cite[Conj.~51]{Handbook}) conjectured the following:

\vspace{0.2cm}

Conjecture $P(X)$: {\it We have $K_i(X)_\bbQ=0$ for every $i \geq1$}.

\vspace{0.2cm}

All the above conjectures hold whenever $\mathrm{dim}(X)\leq 1$; see \cite{Grayson, Harder, Handbook, KS, Tate-motives}. The conjectures $T^l(X)$ and $T^p(X)$ hold\footnote{As explained in \S\ref{sub:divisors} below, whenever $\mathrm{dim}(X)\leq 3$, we have $T^l(X) \Leftrightarrow T^p(X)$ (for every $l \neq p$).} moreover for abelian varieties of dimension $\leq 3$ and for $K3$-surfaces; see \cite{Milne-AIM, Totaro}. Besides these cases (and some other scattered cases), the aforementioned important conjectures remain wide open; consult Theorems \ref{thm:last} and \ref{thm:two} below for a proof of the Tate, Beilinson and Parshin conjectures in several new cases.


Recall from \S\ref{sub:dg} that a {\em differential graded (=dg) category} $\cA$ is a category enriched over dg $k$-vector spaces. As explained in \S\ref{sec:variants}, given a smooth proper dg category $\cA$ in the sense of Kontsevich, the conjectures of Tate, Beilinson and Parshin admit noncommutative analogues $T^l_{\mathrm{nc}}(\cA)$, $T^p_{\mathrm{nc}}(\cA)$, $B_{\mathrm{nc}}(\cA)$, and $P_{\mathrm{nc}}(\cA)$, respectively. Examples of smooth proper dg categories include finite dimensional $k$-algebras of finite global dimension $A$ as well as the canonical dg enhancement $\perf_\dg(X)$ of the category of perfect complexes $\perf(X)$ of smooth proper $k$-schemes $X$ (or, more generally, smooth proper algebraic stacks $\cX$); consult \cite{ICM-Keller, LO}.
\begin{theorem}\label{thm:Thomason}
Given a smooth projective $k$-scheme $X$, we have the equivalences:
\begin{eqnarray*}
T^l(X) \Leftrightarrow T^l_{\mathrm{nc}}(\perf_\dg(X)) && T^p(X) \Leftrightarrow T^p_{\mathrm{nc}}(\perf_\dg(X)) \\
B(X) \Leftrightarrow B_{\mathrm{nc}}(\perf_\dg(X))&& P(X) \Leftrightarrow P_{\mathrm{nc}}(\perf_\dg(X))\,.
\end{eqnarray*}
\end{theorem}  
Theorem \ref{thm:Thomason} shows that the conjectures of Tate, Beilinson and Tate belong not only to the realm of algebraic geometry but also to the broad setting of smooth proper dg categories. Making use of this latter noncommutative viewpoint, we now prove that these celebrated conjectures are invariant under Homological Projective Duality (=HPD); for surveys on HPD we invite the reader to consult \cite{Kuznetsov-ICM,Thomas}.

Let $X$ be a smooth projective $k$-scheme equipped with a line bundle $\cL_X(1)$; we write $X \to \bbP(V)$ for the associated morphism, where $V:=H^0(X,\cL_X(1))^\ast$. Assume that the triangulated category $\perf(X)$ admits a Lefschetz decomposition $\langle \bbA_0, \bbA_1(1), \ldots, \bbA_{i-1}(i-1)\rangle$ with respect to $\cL_X(1)$ in the sense of \cite[Def.~4.1]{Kuznetsov-IHES}. Following \cite[Def.~6.1]{Kuznetsov-IHES}, let $Y$ be the HP-dual of $X$, $\cL_Y(1)$ the HP-dual line bundle, and $Y\to \bbP(V^\ast)$ the morphism associated to $\cL_Y(1)$. Given a linear subspace $L\subset V^\ast$, consider the linear sections $X_L:=X\times_{\bbP(V)}\bbP(L^\perp)$ and $Y_L:=Y \times_{\bbP(V^\ast)}\bbP(L)$. 
\begin{theorem}[HPD-invariance]\label{thm:main}
Let $X$ and $Y$ be as above. Assume that $X_L$ and $Y_L$ are smooth\footnote{The linear section $X_L$ is smooth if and only if the linear section $Y_L$ is smooth; see \cite[page~9]{Kuznetsov-ICM}.}, that $\mathrm{dim}(X_L)=\mathrm{dim}(X)-\mathrm{dim}(L)$ and $\mathrm{dim}(Y_L)=\mathrm{dim}(Y)-\mathrm{dim}(L^\perp)$, and that the conjecture $T^l_{\mathrm{nc}}(\bbA_0^\dg)$, resp. $T^p_{\mathrm{nc}}(\bbA_0^\dg)$, resp. $B_{\mathrm{nc}}(\bbA_0^\dg)$, resp. $P_{\mathrm{nc}}(\bbA_0^\dg)$, holds, where $\bbA^{\mathrm{dg}}_0$ stands for the dg enhancement of $\bbA_0$ induced from $\perf_\dg(X)$. Under these assumptions, we have the equivalence $T^l(X_L) \Leftrightarrow T^l(Y_L)$, resp. $T^p(X_L) \Leftrightarrow T^p(Y_L)$, resp. $B(X_L) \Leftrightarrow B(Y_L)$, resp. $P(X_L) \Leftrightarrow P(Y_L)$.
\end{theorem}
\begin{remark}\label{rk:singular}
\begin{itemize}
\item[(i)] Given a {\em generic} subspace $L \subset V^\ast$, the sections $X_L$ and $Y_L$ are smooth, and $\mathrm{dim}(X_L)=\mathrm{dim}(X)-\mathrm{dim}(L)$ and $\mathrm{dim}(Y_L)=\mathrm{dim}(Y)-\mathrm{dim}(L^\perp)$. 
\item[(ii)] The conjectures $T^l_{\mathrm{nc}}(\bbA_0^{\mathrm{dg}})$, $T^p_{\mathrm{nc}}(\bbA_0^{\mathrm{dg}})$, $B_{\mathrm{nc}}(\bbA_0^{\mathrm{dg}})$, and $P_{\mathrm{nc}}(\bbA_0^{\mathrm{dg}})$, hold, in particular, whenever the triangulated category $\bbA_0$ admits a full exceptional collection.
\item[(iii)] Theorem \ref{thm:main} holds more generally when $Y$ is singular. In this case we need to replace $Y$ by a noncommutative resolution of singularities $\perf_\dg(Y;\cF)$, where $\cF$ stands for a certain sheaf of noncommutative algebras over $Y$ (consult \cite[\S2.4]{Kuznetsov-ICM} for details), and conjecture $T^l(Y)$, resp. $T^p(Y)$, resp. $B(Y)$, resp. $P(Y)$, by its noncommutative analogue $T^l_{\mathrm{nc}}(\perf_\dg(Y;\cF))$, resp. $T^p_{\mathrm{nc}}(\perf_\dg(Y;\cF))$, resp. $B_{\mathrm{nc}}(\perf_\dg(Y;\cF))$, resp. $P_{\mathrm{nc}}(\perf_\dg(Y;\cF))$.
\end{itemize}
\end{remark}
To the best of the author's knowledge, Theorem \ref{thm:main} is new in the literature. In what follows, we illustrate its strength in the case of two important HP-dualities.
\subsection*{Determinantal duality}
Let $U_1$ and $U_2$ be two $k$-vector spaces of dimensions $d_1$ and $d_2$, respectively, with $d_1\leq d_2$, $V:=U_1 \otimes U_2$, and $0 < r< d_1$ an integer. 

Consider the determinantal variety $\cZ^r_{d_1, d_2}\subset \bbP(V)$ defined as the locus of those matrices $U_2 \to U_1^\ast$ with rank $\leq r$. Recall that the determinantal varieties with $r=1$ are the classical Segre varieties. For example, $\cZ^1_{2,2}\subset \bbP^3$ is the quadric surface defined as the zero locus of the $2\times 2$ minor $v_0v_3 - v_1 v_2$. In contrast with the Segre varieties, the determinantal varieties $\cZ^r_{d_1, d_2}$, with $r\geq 2$, are not smooth. The singular locus of $\cZ^r_{d_1, d_2}$ consists of those matrices $U_2 \to U_1^\ast$ with rank $<r$, \ie it agrees with the closed subvariety $\cZ^{r-1}_{d_1, d_2}$. Nevertheless, it is well-known that $\cZ^r_{d_1, d_2}$ admits a canonical Springer resolution of singularities $\cX^r_{d_1, d_2} \to \cZ^r_{d_1, d_2}$, which comes equipped with a projection $q\colon \cX^r_{d_1, d_2} \to \mathrm{Gr}(r, U_1)$ to the Grassmannian of $r$-dimensional subspaces in $U_1$. Following \cite[\S3.3]{Marcello}, the category $\perf(X)$, with $X:=\cX^r_{d_1, d_2}$, admits a Lefschetz decomposition $\langle \bbA_0, \bbA_1(1), \ldots, \bbA_{d_2 r - 1}(d_2 r -1)\rangle$, where $\bbA_0=\bbA_1= \cdots = \bbA_{d_2 r -1}= q^\ast(\perf(\mathrm{Gr}(r,U_1)))\simeq \perf(\mathrm{Gr}(r,U_1))$. 
\begin{proposition}\label{prop:key}
The following conjectures hold: 
$$T^l_{\mathrm{nc}}(\bbA_0^{\mathrm{dg}})\quad T^p_{\mathrm{nc}}(\bbA_0^{\mathrm{dg}}) \quad B_{\mathrm{nc}}(\bbA_0^{\mathrm{dg}}) \quad P_{\mathrm{nc}}(\bbA_0^{\mathrm{dg}})\,.$$
\end{proposition}
Dually, consider the variety $\cW^r_{d_1, d_2}\subset \bbP(V^\ast)$, defined as the locus of those matrices $U^\ast_2 \to U_1$ with corank $\geq r$, and the associated Springer resolutions of singularities $Y:=\cY^r_{d_1, d_2} \to \cW^r_{d_1, d_2}$. As proved\footnote{In \cite[Prop.~3.4 and Thm.~3.5]{Marcello} the authors worked over an algebraically closed field of characteristic zero. However, the same proof holds {\em mutatis mutandis} over $k=\bbF_q$. Simply replace the reference \cite{Kapranov} concerning the existence of a full strong exceptional collection on $\perf(\mathrm{Gr}(r, U_1))$ by the reference \cite[Thm.~1.3]{VdB} concerning the existence of a tilting bundle on $\perf(\mathrm{Gr}(r, U_1))$. The author is grateful to Marcello Bernardara for discussions concerning this issue.} in \cite[Prop.~3.4 and Thm.~3.5]{Marcello}, $X$ and $Y$ are HP-dual to each other. 

Given a generic linear subspace $L \subseteq V^\ast$, consider the smooth linear sections $X_L$ and $Y_L$; note that whenever $\bbP(L^\perp)$ does not intersects the singular locus of $\cZ^r_{d_1, d_2}$, we have $X_L=\bbP(L^\perp) \cap \cZ^r_{d_1, d_2}$. Theorem \ref{thm:main} yields the following result:
\begin{corollary}\label{cor:last}
We have the following equivalences: 
$$
T^l(X_L)\Leftrightarrow T^l(Y_L) \quad T^p(X_L)\Leftrightarrow T^p(Y_L) \quad
B(X_L)\Leftrightarrow B(Y_L) \quad P(X_L)\Leftrightarrow P(Y_L)\,.
$$
\end{corollary}
By construction, $\mathrm{dim}(X)= r(d_1+ d_2 -r)-1$ and $\mathrm{dim}(Y)= r(d_1-d_2-r) + d_1 d_2 -1$. Consequently, we have $\mathrm{dim}(X_L)= r(d_1+d_2-r) -1 - \mathrm{dim}(L)$ and $\mathrm{dim}(Y_L)= r(d_1 - d_2 - r) -1 +\mathrm{dim}(L)$. Since the Tate, Beilinson and Parshin conjectures hold in dimensions $\leq 1$, we hence obtain from Corollary \ref{cor:last} the following result:
\begin{theorem}[Linear sections of determinantal varieties]\label{thm:last}
Let $X_L$ and $Y_L$ be smooth linear sections of determinantal varieties as in Corollary \ref{cor:last}.
\begin{itemize}
\item[(i)] Whenever $r(d_1 + d_2 -r) -1 - \mathrm{dim}(L)\leq 1$, the conjectures $T^l(Y_L)$, $T^p(Y_L)$, $B(Y_L)$, and $P(Y_L)$, hold.
\item[(ii)] Whenever $r(d_1 -d_2 -r) -1 + \mathrm{dim}(L)\leq 1$, the conjectures $T^l(X_L)$, $T^p(X_L)$, $B(X_L)$, and $P(X_L)$, hold.
\end{itemize}
\end{theorem}
To the best of the author's knowledge, Theorem \ref{thm:last} is new in the literature. It proves the Tate, Beilinson and Parshin conjectures in several new cases. Here are two families of examples: 
\begin{example}[Segre varieties]\label{ex:Segre}
Let $r=1$. Thanks to Theorem \ref{thm:last}(ii), whenever $d_1 - d_2 -2 +\mathrm{dim}(L) \leq 1$, the conjectures $T^l(X_L)$, $T^p(X_L)$, $B(X_L)$, and $P(X_L)$, hold. In all these cases, $X_L$ is a linear section of the Segre variety $\cZ^1_{d_1, d_2}$ and the dimension of $X_L$ is $2(d_2 - \mathrm{dim}(L))$ or $2(d_2 - \mathrm{dim}(L))+1$. Therefore, for example, by letting $d_2 \to \infty$ and by keeping $\mathrm{dim}(L)$ fixed, we obtain infinitely many new examples of smooth projective $k$-schemes $X_L$, of arbitrary high dimension, satisfying the Tate, Beilinson and Parshin conjectures.
\end{example}
\begin{subexample}
Let $r=1$, $d_1=4$, and $d_2=2$. In this particular case, the Segre variety $\cZ^1_{4, 2} \subset \bbP^7$ agrees with the rational normal $4$-fold scroll $S_{1,1,1,1}$; see \cite[Ex.~8.27]{Harris}. Choose a generic linear subspace $L \subset V^\ast$ of dimension $1$ such that the hyperplane $\bbP(L^\perp)\subset \bbP^7$ does not contain any $3$-plane of the rulling of $S_{1,1,1,1}$. By combining Example \ref{ex:Segre} with \cite[Prop.~2.5]{Faenzi}, we conclude that the rational normal $3$-fold scroll $X_L=S_{1,1,2}$ satisfies the Tate, Beilinson and Parshin conjectures.
\end{subexample}
\begin{example}[Square matrices]\label{ex:square}
Let $d_1=d_2=d$. Thanks to Theorem \ref{thm:last}(ii), whenever $-r^2-1 + \mathrm{dim}(L)\leq 1$, the conjectures $T^l(X_L), T^p(X_L), B(X_L), P(X_L)$ hold. In all these cases the dimension of $X_L$ is $2(dr - \mathrm{dim}(L))$ or $2(dr - \mathrm{dim}(L))+1$. Therefore, for example, by letting $d\to \infty$ and by keeping $r$ and $\mathrm{dim}(L)$ fixed, we obtain infinitely many new examples of smooth projective $k$-schemes $X_L$, of arbitrary high dimension, satisfying the Tate, Beilinson and Parshin conjectures.
\end{example}
\begin{subexample}
Let $d_1=d_2=3$ and $r=2$. In this particular case, the determinantal variety $\cZ^2_{3,3}\subset \bbP^8$ has dimension $7$ and its singular locus is the $4$-dimensional Segre variety $\cZ^1_{3,3} \subset \cZ^2_{3,3}$. Given a generic linear subspace $L\subset V^\ast$ of dimension $5$, the associated smooth linear section $X_L$ is $2$-dimensional and, thanks to Example \ref{ex:square}, it satisfies the Tate, Beilinson and Parshin conjectures. Note that since $\mathrm{codim}(L^\perp)=5>4 = \mathrm{dim}(\cZ^1_{3,3})$, the subspace $\bbP(L^\perp)\subset \bbP^8$ does not intersects the singular locus $\cZ^1_{3,3}$ of $\cZ^2_{3,3}$. Therefore, for all the above choices of $L$, the associated surface $X_L$ is a linear section of the determinantal variety $\cZ^2_{3,3}$.
\end{subexample}
\subsection*{Veronese-Clifford duality}
Let $W$ be a $k$-vector space of dimension $d$ and $X$ the associated projective space $\bbP(W)$ equipped with the double Veronese embedding $\bbP(W) \to \bbP(S^2W), [w] \mapsto [w\otimes w]$. Consider the Beilinson's full exceptional collection $\perf(X)=\langle \cO_X(-1), \cO_X, \cO_X(1), \ldots, \cO_X(d-2)\rangle$ (see \cite{Beilinson}) and set $i:=\lceil d/2 \rceil$ and 
\begin{eqnarray*}
\bbA_0=\bbA_1=\cdots = \bbA_{i-2}:=\langle \cO_X(-1), \cO_X\rangle & \bbA_{i-1}:=\begin{cases}
\langle \cO_X(-1), \cO_X\rangle & \text{if}\,\,d=2i\\
\langle \cO_X(-1)\rangle & \text{if}\,\,d=2i-1.
\end{cases} &
\end{eqnarray*}
Under these notations, the category $\perf(X)$ admits the Lefschetz decomposition $\langle \bbA_0, \bbA_1(1), \ldots, \bbA_{i-1}(i-1)\rangle$ with respect to the line bundle $\cL_X(1)=\cO_X(2)$. Remark \ref{rk:singular}(ii) hence implies the conjectures $T^l_{\mathrm{nc}}(\bbA_0^{\mathrm{dg}})$, $T^p_{\mathrm{nc}}(\bbA_0^{\mathrm{dg}})$, $B_{\mathrm{nc}}(\bbA_0^{\mathrm{dg}})$, and $P_{\mathrm{nc}}(\bbA_0^{\mathrm{dg}})$.

Let $\cH:=X\times_{\bbP(S^2W)} \cQ \subset X \times \bbP(S^2 W^\ast)$ be the universal hyperplane section, where $\cQ \subset \bbP(S^2 W) \times \bbP(S^2W^\ast)$ stands for the incidence quadric. By construction, the projection $q\colon \cH \to \bbP(S^2 W^\ast)$ is a flat quadric fibration. As proved in \cite[Thm.~5.4]{Kuznetsov-quadrics} (see also \cite[Thm.~2.3.6]{Bernardara}) the HP-dual $Y$ of $X$ is given by $\perf_\dg(\bbP(S^2 W^\ast); \mathcal{C}l_0(q))$ (see Remark \ref{rk:singular}(iii)), where $\mathcal{C}l_0(q)$ stands for the sheaf of even Clifford algebras associated to $q$. 

Let $L \subset S^2W^\ast$ be a generic linear subspace. On the one hand, $X_L$ corresponds to the smooth complete intersection of the $\mathrm{dim}(L)$ quadric hypersurfaces in $\bbP(W)$ parametrized by $L$. On the other hand, $Y_L$ is given by $\perf_\dg(\bbP(L); \mathcal{C}l_0(q)_{|L})$. Theorem \ref{thm:main} yields the following result:
\begin{corollary}\label{cor:main}
We have the following equivalences:
\begin{eqnarray*}
T^l(X_L) \Leftrightarrow T^l_{\mathrm{nc}}(\perf_\dg(\bbP(L); \mathcal{C}l_0(q)_{|L})) && T^p(X_L) \Leftrightarrow T^p_{\mathrm{nc}}(\perf_\dg(\bbP(L); \mathcal{C}l_0(q)_{|L})) \\
B(X_L) \Leftrightarrow B_{\mathrm{nc}}(\perf_\dg(\bbP(L); \mathcal{C}l_0(q)_{|L})) && P(X_L) \Leftrightarrow P_{\mathrm{nc}}(\perf_\dg(\bbP(L); \mathcal{C}l_0(q)_{|L}))\,.
\end{eqnarray*}
\end{corollary}
Recall that the space of quadrics $\bbP(S^2 W^\ast)$ comes equipped with a canonical filtration $\Delta_d \subset \cdots \subset \Delta_2 \subset \Delta_1 \subset \bbP(S^2W^\ast)$, where $\Delta_i$ stands for the closed subscheme of those singular quadrics of corank $\geq i$. 
\begin{theorem}[Intersection of two quadrics]\label{thm:two}
Let $X_L$ be as in Corollary \ref{cor:main}. Assume that $\mathrm{dim}(L)=2$, that $\bbP(L) \cap \Delta_2 =\emptyset$, and that $p\neq 2$ when $d$ is odd. Under these assumptions, the conjectures $T^l(X_L)$, $T^p(X_L)$, $B(X_L)$, and $P(X_L)$, hold.
\end{theorem}
\begin{remark}[Intersection of even-dimensional quadrics]
In the case of an intersection $X_L$ of (several) even-dimensional quadrics, we prove in Theorem \ref{prop:2} below that the conjectures $T^l(X_L)$ (for every $l\neq 2$), $T^p(X_L)$, $B(X_L)$, and $P(X_L)$, are equivalent to the corresponding conjectures for the discriminant $2$-fold cover of the projective space $\bbP(L)$. To the best of the author's knowledge, this (geometric) result is new in the literature.
\end{remark}
The proof of Theorem \ref{thm:two} is based on the solution of the corresponding noncommutative conjectures of Corollary \ref{cor:main}; consult \S\ref{sec:proof} for details. In what concerns Tate's conjecture, an alternative (geometric) proof, based on the notion of variety of maximal planes, was obtained by Reid\footnote{Reid also assumed in {\em loc. cit.} that $\bbP(L) \cap \Delta_2 =\emptyset$; see \cite[Def.~1.9]{Reid}.} in the early seventies; see \cite[Thms.~3.14 and 4.14]{Reid}. Therein, Reid proved the Hodge conjecture but, as Kahn kindly informed me, a similar proof works for the Tate conjecture. In what concerns the Beilinson and Parshin conjectures, Theorem \ref{thm:two} is new in the literature.
\subsection*{Strong form of the Tate conjecture}
Given a smooth projective $k$-scheme $X$, let us write $\mathrm{ord}_{s=i} \zeta(X,s)$ for the order of the pole of the Hasse-Weil zeta function $\zeta(X,s)$ of $X$ at $i \in \{0, 1, \ldots, \mathrm{dim}(X)\}$. In the sixties, Tate \cite{Tate} also conjectured the following strong form of the Tate conjecture\footnote{As proved in \cite[Thm.~2.9]{Tate-motives}, resp. \cite[Thm.~1.11]{Milne-AIM}, the strong form of the Tate conjecture $ST(X)$ implies the Tate conjecture $T^l(X)$ (for every $l\neq p$), resp. the $p$-version of the Tate conjecture $T^p(X)$.}:
\vspace{0.2cm}

Conjecture $ST(X)$: {\it The equality $\mathrm{ord}_{s=i} \zeta(X,s)= \mathrm{dim}\, \cZ^i(X)_\bbQ/_{\!\sim \mathrm{num}}$ holds}.

\vspace{0.2cm}

Let us write $\cZ^\ast(X)_\bbQ/_{\!\sim l\text{-}\mathrm{adic}}$, resp. $\cZ^\ast(X)_\bbQ/_{\!\sim\mathrm{crys}}$, for the quotient of $\cZ^\ast(X)_\bbQ$ with respect to the $l$-adic homological equivalence relation, resp. crystalline homological equivalence relation. Note that Beilinson's conjecture $B(X)$ implies that $\cZ^\ast(X)_\bbQ/_{\!\sim l\text{-}\mathrm{adic}}=\cZ^\ast(X)_\bbQ/_{\!\sim\mathrm{crys}}=\cZ^\ast(X)_\bbQ/_{\!\sim \mathrm{num}}$. Therefore, making use of \cite[Thm.~2.9]{Tate-motives}, resp.  \cite[Thm.~1.11]{Milne-AIM}, we conclude that $T^l(X) + B(X) \Rightarrow ST(X)$, resp. $T^p(X) + B(X) \Rightarrow ST(X)$. This implies that the strong form of the Tate conjecture also holds in the several new cases provided by Theorems \ref{thm:last} and \ref{thm:two}.
\subsection*{Tate, Beilinson and Parshin conjectures for stacks}
Theorem \ref{thm:Thomason} allow us to easily extend the original conjectures of Tate, Beilinson and Parshin from smooth projective $k$-schemes $X$ to smooth proper algebraic $k$-stacks $\cX$ by setting 
\begin{eqnarray*}
?(\cX):= ?_{\mathrm{nc}}(\perf_\dg(\cX)) & \mathrm{with} & ?\in \{T^l, T^p, B, P\}\,.
\end{eqnarray*}
The next result proves these extended conjectures for ``low-dimensional'' orbifolds:
\begin{theorem}\label{thm:orbifold}
Let $G$ be a finite group of order $s$, $X$ a smooth projective $k$-scheme equipped with a $G$-action, and $\cX:=[X/G]$ the associated global orbifold. If $p \nmid s$, then we have the following implications
\begin{eqnarray}
& \sum_{\sigma \subseteq G} T^l(X^\sigma \times \mathrm{Spec}(k[\sigma]))  \Rightarrow  T^l(\cX)&  (\mathrm{for}\,\mathrm{every}\,\, l\nmid s) \label{eq:implication-main1}\\
& \sum_{\sigma \subseteq G} T^p(X^\sigma \times \mathrm{Spec}(k[\sigma]))  \Rightarrow  T^p(\cX) & \label{eq:implication-main2}\\
& \sum_{\sigma \subseteq G} B(X^\sigma \times \mathrm{Spec}(k[\sigma]))  \Rightarrow  B(\cX)  & \label{eq:implication-main3}\\
& \,\,\sum_{\sigma \subseteq G} P(X^\sigma \times \mathrm{Spec}(k[\sigma])) \Rightarrow P(\cX)\,, & \label{eq:implication-main4}
\end{eqnarray}
where $\sigma$ is an arbitrary cyclic subgroup of $G$. Moreover, whenever $s\mid (q-1)$, resp. $\mathrm{dim}(X)\leq 3$, the $k$-schemes $X^\sigma \times \mathrm{Spec}(k[\sigma])$ in \eqref{eq:implication-main1}-\eqref{eq:implication-main4}, resp. in \eqref{eq:implication-main1}-\eqref{eq:implication-main2}, can be replaced by the $k$-schemes $X^\sigma$.
\end{theorem}
Note that the assumption $s\mid (q-1)$ implies that $p\nmid s$.
\begin{corollary}\label{cor:examples} 
\begin{itemize}
\item[(i)] Assume that $p\nmid s$. If $\mathrm{dim}(X)\leq 1$, then the conjectures $T^l(\cX)$ (for every $l\nmid s$), $T^p(\cX)$, $B(\cX)$, and $P(\cX)$, hold. 
\item[(ii)]  Assume\footnote{In the particular case of the ($p$-version of the) Tate conjecture, it suffices to assume that $p\nmid s$.} that $s\mid (q-1)$. If $\mathrm{dim}(X)=2$, then $T^l(X) \Rightarrow T^l(\cX)$ (for every $l\nmid s$), $T^l(X) \Rightarrow T^l(\cX)$,  $B(X) \Rightarrow B(\cX)$, and $P(X) \Rightarrow P(\cX)$.
\end{itemize}
\end{corollary}
\begin{example}
Let $X$ be an abelian surface equipped with the $\bbZ/2$-action $a \mapsto-a$. Since the conjectures $T^l(X)$ and $T^p(X)$ hold, Corollary \ref{cor:examples}(ii) implies that the conjectures $T^l(\cX)$ (for every $l\neq 2$) and $T^p(\cX)$ also hold.
\end{example}
%
We finish this section with the following ``twisted'' version of Corollary \ref{cor:examples}:
\begin{theorem}\label{prop:twisted}
Let $G$ be a finite group of order $s$, $X$ a smooth projective $k$-scheme equipped with a $G$-action, $\cX:=[X/G]$ the associated global orbifold, and $\cF$ a $G$-equivariant sheaf of Azumaya algebras over $S$ of rank $r$. Assume that $s\mid (q-1)$.
\begin{itemize}
\item[(i)] If $\mathrm{dim}(X)\leq 1$, then the conjectures $T^l(\cX;\cF)$ (for every $l\nmid sr$), $T^p(\cX;\cF)$, $B(\cX;\cF)$, and $P(\cX;\cF)$, hold.
\item[(ii)] If the $G$-action is faithful and $\mathrm{dim}(X)=2$, then $T^l(X) \Rightarrow T^l(\cX;\cF)$ (for every $l\nmid sr$), $T^p(X) \Rightarrow T^p(\cX;\cF)$, $B(X) \Rightarrow B(\cX;\cF)$, and $P(X) \Rightarrow P(\cX;\cF)$.
\end{itemize}
\end{theorem}
\section{Preliminaries}\label{sec:preliminaries}
Throughout the article, $k:=\bbF_q$ is a finite field of characteristic $p$ with $q=p^n$.
\subsection{Dg categories}\label{sub:dg}
For a survey on dg categories, we invite the reader to consult Keller's ICM address \cite{ICM-Keller}. Let $(\cC(k),\otimes, k)$ be the category of dg $k$-vector spaces. A {\em differential
  graded (=dg) category $\cA$} is a category enriched over $\cC(k)$
and a {\em dg functor} $F\colon\cA\to \cB$ is a functor enriched over
$\cC(k)$. In what follows, we will write $\dgcat(k)$ for the category of (essentially small) dg categories and dg functors.

Let $\cA$ be a dg category. The opposite dg category $\cA^\op$ has the
same objects and $\cA^\op(x,y):=\cA(y,x)$. A {\em right dg
  $\cA$-module} is a dg functor $M\colon \cA^\op \to \cC_\dg(k)$ with values
in the dg category $\cC_\dg(k)$ of dg $k$-vector spaces. Following \cite[\S3.2]{ICM-Keller}, the derived category $\cD(\cA)$ of $\cA$ is defined as the localization of the category of right dg $\cA$-modules $\cC(\cA)$ with respect to the objectwise quasi-isomorphisms. In what follows, we will write $\cD_c(\cA)$ for the triangulated subcategory of compact objects.

A dg functor $F\colon\cA\to \cB$ is called a {\em Morita equivalence} if it
induces an equivalence on derived categories $\cD(\cA) \simeq 
\cD(\cB)$; see \cite[\S4.6]{ICM-Keller}. As explained in
\cite[\S1.6]{book}, the category $\dgcat(k)$ admits a Quillen model
structure whose weak equivalences are the Morita equivalences. Let us
denote by $\Hmo(k)$ the associated homotopy category.

The {\em tensor product $\cA\otimes\cB$} of dg categories is defined
as follows: the set of objects is $\mathrm{obj}(\cA)\times \mathrm{obj}(\cB)$ and
$(\cA\otimes\cB)((x,w),(y,z)):= \cA(x,y) \otimes \cB(w,z)$. As
explained in \cite[\S2.3]{ICM-Keller}, this construction gives rise to
a symmetric monoidal structure on $\dgcat(k)$ which descends to the homotopy category
$\Hmo(k)$. 

A {\em dg $\cA\text{-}\cB$-bimodule} is a dg functor
$\mathrm{B}\colon \cA\otimes \cB^\op \to \cC_\dg(k)$ or, equivalently, a
right dg $(\cA^\op \otimes \cB)$-module. A standard example is the dg
$\cA\text{-}\cB$-bimodule
\begin{eqnarray}\label{eq:bimodule2}
{}_F\cB:\cA\otimes \cB^\op \to \cC_\dg(k) && (x,z) \mapsto \cB(z,F(x))
\end{eqnarray}
associated to a dg functor $F:\cA\to \cB$. Following Kontsevich \cite{Miami,finMot,IAS}, a dg category $\cA$ is called {\em smooth} if the dg $\cA\text{-}\cA$-bimodule ${}_{\id}\cA$ belongs to the category $\cD_c(\cA^\op\otimes \cA)$ and {\em proper} if $\sum_i \mathrm{dim}\, H^i\cA(x,y)< \infty$ for any pair of objects $(x,y)$. 

\subsection{Additive invariants}\label{sub:additive}
Let $\cA$ and $\cB$ be two dg categories and $\mathrm{B}$ a dg $\cA\text{-}\cB$-bimodule. Consider the following dg category $T(\cA,\cB;\mathrm{B})$: the set of objects is $\mathrm{obj}(\cA)\amalg \mathrm{obj}(\cB)$, the dg $k$-vector spaces of morphisms are given as follows 
$$ T(\cA, \cB; \mathrm{B})(x,y):=\begin{cases}
\cA(x,y) & \text{if}\,\,\, x, y \in \cA \\
\cB(x,y) & \text{if}\,\,\, x, y \in \cB \\
\mathrm{B}(x,y) & \text{if}\,\,\, x \in \cA\,\,\text{and}\,\, y \in \cB \\
0 & \text{if}\,\,\, x \in \cB\,\,\text{and}\,\, y \in \cA \,,\\
\end{cases}
$$
and the composition law is induced by the composition law of $\cA$ and $\cB$ and by the dg $\cA\text{-}\cB$-bimodule structure of $\mathrm{B}$. Note that, by construction, we have canonical dg functors $\iota_\cA \colon \cA \to T(\cA, \cB;\mathrm{B})$ and $\iota_\cB\colon \cB \to T(\cA, \cB; \mathrm{B})$. 

Recall from \cite[Def.~2.1]{book} that a functor $E\colon \dgcat(k) \to \mathrm{D}$, with values in an additive category, is called an {\em additive invariant} if it satisfies the following conditions:
\begin{itemize}
\item[(i)] It sends the Morita equivalences to isomorphisms.
\item[(ii)] Given $\cA$, $\cB$, and $\mathrm{B}$, as above, the dg functos $\iota_\cA$ and $\iota_\cB$ induce an isomorphism
$$ E(\cA) \oplus E(\cB) \stackrel{\sim}{\too} E(T(\cA,\cB;\mathrm{B}))\,.$$
\end{itemize}
Let us write $\rep(\cA,\cB)$ for the full triangulated subcategory of $\cD(\cA^\op \otimes \cB)$ consisting of those dg $\cA\text{-}\cB$-modules $\mathrm{B}$ such that for every object $x \in \cA$ the associated right dg $\cB$-module $\mathrm{B}(x,-)$ belongs to $\cD_c(\cB)$. As explained in \cite[\S1.6.3]{book}, there is a
  natural bijection between $\Hom_{\Hmo(k)}(\cA,\cB)$ and the set of
  isomorphism classes of the category $\rep(\cA,\cB)$. Under this bijection, the composition law of $\Hmo(k)$ corresponds to the (derived) tensor product of bimodules. Therefore, since the dg $\cA\text{-}\cB$-bimodules
  \eqref{eq:bimodule2} belong to $\rep(\cA,\cB)$, we have the following symmetric monoidal functor:
\begin{eqnarray}\label{eq:functor1}
\dgcat(k)\too \Hmo(k) & \cA \mapsto \cA & (\cA\stackrel{F}{\to} \cB) \mapsto {}_F \cB\,.
\end{eqnarray}
The {\em additivization} of $\Hmo(k)$ is the additive category
$\Hmo_0(k)$ with the same objects as $\Hmo(k)$ and with abelian groups of morphisms $\Hom_{\Hmo_0(k)}(\cA,\cB)$ given by the Grothendieck group
$K_0\rep(\cA,\cB)$ of the triangulated category $\rep(\cA,\cB)$. As explained in \cite[\S2.3]{book}, the following composition
\begin{eqnarray}\label{eq:universal}
U\colon \dgcat(k) \stackrel{\eqref{eq:functor1}}{\too} \Hmo(k) \too \Hmo_0(k) && \cA \mapsto \cA \quad (\cA \stackrel{F}{\to} \cB) \mapsto [{}_F \cB]
\end{eqnarray}
is the {\em universal} additive invariant. Moreover, the symmetric monoidal structure of $\Hmo(k)$ extends to $\Hmo_0(k)$, making the above functor \eqref{eq:universal} symmetric monoidal.
\subsection{Noncommutative motives}
For a book, resp. survey, on noncommutative motives, we invite the reader to consult \cite{book}, resp. \cite{survey}. Given a commutative ring $R$, recall from \cite[\S4.1]{book} that the category of {\em noncommutative Chow motives} $\NChow(k)_R$ (with $R$-coefficients) is defined as the idempotent completion of the full subcategory of $\Hmo_0(k)_R$ consisting of those objects $U(\cA)_R$ with $\cA$ a smooth proper dg category. As explained in {\em loc. cit.}, the category $\NChow(k)_R$ is $R$-linear, additive, and rigid symmetric monoidal. Moreover, we have natural isomorphisms:
\begin{equation}\label{eq:isos}
\Hom_{\NChow(k)_R}(U(\cA)_R, U(\cB)_R):= K_0(\rep(\cA^\op \otimes \cB))_R \simeq K_0(\cA^\op \otimes \cB)_R\,.
\end{equation}
Given a $R$-linear, additive, rigid symmetric monoidal category $(\cC,\otimes, {\bf 1})$, its $\cN$-ideal is defined as follows ($\mathrm{tr}(g\circ f)$ stands for the categorical trace of $g\circ f$):
\begin{equation*}\label{eq:N}
\cN(a, b) := \{f \in \Hom_\cC(a, b) \,\,|\,\, \forall g \in \Hom_\cC(b, a)\,\, \mathrm{we}\,\,\mathrm{have}\,\, \mathrm{tr}(g\circ f) =0 \}\,.
\end{equation*}
Under these notations, recall from \cite[\S4.6]{book} that the category of {\em noncommutative numerical motives} $\NNum(k)_R$ (with $R$-coefficients) is defined as the idempotent completion of the quotient category $\NChow(k)_R/\cN$.
\subsection{Tate conjecture for divisors}\label{sub:divisors}
Let $X$ be a smooth projective $k$-scheme of dimension $d$. Given a prime number $l\neq p$, consider the Tate conjecture for divisors:

\vspace{0.2cm}

Conjecture $T^{l,1}(X)$: {\it The cycle class map \eqref{eq:class-map} with $\ast=1$ is surjective}.

\vspace{0.2cm}

As proved in \cite[Prop.~4.1]{Tate-motives}, we have the implication $T^{l,1}(X) \Rightarrow T^{l,d-1}(X)$. Consequently, whenever $\mathrm{dim}(X)\leq 3$, we conclude that $T^l(X) \Leftrightarrow T^{l,1}(X)$.

Consider also the $p$-version of the Tate conjecture for divisors:

\vspace{0.2cm}

Conjecture $T^{p,1}(X)$: {\it The cycle class map \eqref{eq:cycle-crystalline} with $\ast=1$ is surjective}.

\vspace{0.2cm}

As proved in \cite[Prop.~4.1]{Morrow}, we have $T^{p,1}(X) \Leftrightarrow T^{l,1}(X)$ (for every $l\neq p$). Moreover, a proof similar to the one of \cite[Prop.~5.1]{Tate-motives}, with the commutative diagram (2.3) of \cite{Tate-motives} replaced by the commutative diagram of \cite[page~25]{Morrow}, shows that $T^{p,1}(X) \Rightarrow T^{p, d-1}(X)$. Consequently, whenever $\mathrm{dim}(X)\leq 3$, we conclude that $T^p(X) \Leftrightarrow T^{p,1}(X)$. This implies that whenever $\mathrm{dim}(X)\leq 3$, we have the equivalence $T^l(X) \Leftrightarrow T^p(X)$ (for every $l\neq p$).

\section{Noncommutative conjectures}\label{sec:variants}
Throughout this section, $\cA$ denotes a smooth proper dg category.
\subsection*{Noncommutative Tate conjecture}
Given a prime number $l\neq p$, consider the following abelian groups 
\begin{eqnarray}\label{eq:groups11}
\Hom\left(\bbZ(l^\infty), \pi_{-1} L_{KU}K(\cA\otimes_{\bbF_q} \bbF_{q^m})\right) && m \geq 1\,,
\end{eqnarray}
where $\bbZ(l^\infty)$ stands for the Pr\"ufer $l$-group, $K(\cA\otimes_{\bbF_q} \bbF_{q^m})$ for the algebraic $K$-theory spectrum of the dg category $\cA\otimes_{\bbF_q} \bbF_{q^m}$, and $L_{KU}K(\cA\otimes_{\bbF_q} \bbF_{q^m})$ for the Bousfield localization of $K(\cA\otimes_{\bbF_q} \bbF_{q^m})$ with respect to topological complex $K$-theory $KU$. Note that the abelian groups \eqref{eq:groups11} can, alternatively, be defined as the $l$-adic Tate module of the abelian groups $\pi_{-1} L_{KU}K(\cA\otimes_{\bbF_q} \bbF_{q^m}), m \geq 1$. Under these notations, Tate's conjecture admits the following noncommutative analogue:

\vspace{0.2cm}

Conjecture $T^l_{\mathrm{nc}}(\cA)$: {\it The abelian groups \eqref{eq:groups11} are trivial}.

\begin{remark} Note that the conjecture $T^l_{\mathrm{nc}}(\cA)$ holds, for example, whenever the abelian groups $\pi_{-1} L_{KU}K(\cA\otimes_{\bbF_q} \bbF_{q^m}), m \ge 1$, are finitely generated. 
\end{remark}
\subsection*{Noncommutative $p$-version of the Tate conjecture}
By construction, the topological Hochschild homology $THH(\cA)$ of $\cA$ carries a canonical $S^1$-action. This leads naturally to the spectrum of homotopy orbits $THH(\cA)_{hS^1}$, to the spectrum of homotopy fixed-points $TC^-(\cA):=THH(\cA)^{h S^1}$, and also to the Tate construction $TP(\cA):=THH(\cA)^{t S^1}$. As explained in \cite[Cor.~I.4.3]{NS}, these spectra are related by the following cofiber sequence
\begin{equation}\label{eq:cofiber}
\Sigma \,THH(\cA)_{hS^1} \stackrel{\mathrm{N}}{\too} THH(\cA)^{hS^1} \stackrel{\mathrm{can}}{\too} THH(\cA)^{tS^1}\,, 
\end{equation}
where $\mathrm{N}$ stands for the norm map. It is well-known that the abelian groups $THH_\ast(\cA)$ are $k$-linear. Hence, after inverting $p$, we have $\Sigma \,THH(\cA)_{hS^1}[1/p]\simeq \ast$. Consequently, the above cofiber sequence \eqref{eq:cofiber} leads to a canonical isomorphism:
\begin{equation}\label{eq:canonical}
\mathrm{can}\colon TC_0^-(\cA)_{1/p} \stackrel{\sim}{\too} TP_0(\cA)_{1/p}\,.
\end{equation}
It is also well-known that the spectrum $THH(\cA)$ is {\em bounded below}, i.e. there exists a integer $m\gg 0$ such that $THH_i(\cA)=0$ for every $i<m$. This follows, for example, from B\"okstedt's celebrated computation $THH_\ast(k)\simeq k[u]$ (where the variable $u$ is of degree $2$) and from the fact that $THH_\ast(\cA)$ is a dualizable $THH_\ast(k)$-module. Since the abelian groups $THH_\ast(\cA)$ are $k$-linear, the spectrum $THH(\cA)$ is moreover $p$-complete. Making use of \cite[Lem.~II 4.2]{NS}, we hence obtain a ``cyclotomic Frobenius'' (which is defined before inverting $p$):
\begin{equation}\label{eq:cyclotomic}
\varphi_p\colon TC^-_0(\cA)_{1/p} \too TP_0(\cA)_{1/p}\,.
\end{equation} 
Let $\varphi:=\varphi_p \circ \mathrm{can}^{-1}$ be the associated endomorphism of $TP_0(\cA)_{1/p}$. It is also well-known that $TP_0(\cA)_{1/p}$ is a (finitely generated) module over $TP_0(k)_{1/p}\simeq K$, i.e. a (finite-dimensional) $K$-vector space. Moreover, the endomorphism $\varphi$ is $\varsigma$-semilinear with respect to the automorphism $\varsigma\colon K\to K$ that acts as $\lambda \mapsto \lambda^p$ on $k$. Hence, $\varphi^n$ becomes a $K$-linear endomorphism of $TP_0(\cA)_{1/p}$.

Recall from \cite[Prop.~4.2]{positive} that the assignment $\cA\mapsto TP_0(\cA)_{1/p}$ gives rise to a $K$-linear functor with values in the category of $K$-vector spaces:
\begin{equation}\label{eq:TP}
TP_0(-)_{1/p}\colon \NChow(k)_K \too \mathrm{Vect}(K)\,.
\end{equation}
This leads to the induced $K$-linear homomorphism: 
$$
K_0(\cA)_K \simeq \Hom(U(k)_K,U(\cA)_K) \stackrel{\theta}{\too} \Hom(TP_0(k)_{1/p}, TP_0(\cA)_{1/p})\simeq TP_0(\cA)_{1/p}\,.
$$
\begin{lemma}\label{lem:factorization}
The preceding homomorphism $\theta$ take values in the $K$-linear subspace $TP_0(\cA)_{1/p}^{\varphi^n}$ of those elements which are fixed by the $K$-linear endomorphism $\varphi^n$.
\end{lemma}
\begin{proof}
On the one hand, the $K$-linear endomorphisms $\varphi^n\colon TP_0(\cA)_{1/p} \to TP_0(\cA)_{1/p}$ (parametrized by the smooth proper dg categories $\cA$) give rise to a natural transformation of the above functor \eqref{eq:TP}. On the other hand, thanks to the enriched Yoneda lemma, the $K$-linear natural transformations from the following functor
$$ K_0(-)_K\simeq \Hom_{\NChow(k)_K}(U(k)_K,-)\colon \NChow(k)_K \too \mathrm{Vect}(K) $$
to the above functor \eqref{eq:TP} are in one-to-one correspondence with the elements of $TP_0(k)_{1/p}\simeq K$. Under this bijection, the identity element $1 \in K$ corresponds to the above homomorphisms $\theta$. Therefore, in order to prove Lemma \ref{lem:factorization}, it suffices to show that the endomorphism $\varphi^n\colon TP_0(k)_{1/p} \to TP_0(k)_{1/p}$ sends $1$ to $1$. This follows from the following explicit descriptions
\begin{eqnarray*}
\mathrm{can}\colon W(k)[u,v]/(uv-p) \too  W(k)[\delta, \delta^{-1}] && u \mapsto p\delta \quad v \mapsto \delta^{-1} \\
\varphi_p\colon W(k)[u,v]/(uv-p) \too W(k)[\delta, \delta^{-1}] && u \mapsto \delta \quad v \mapsto p \delta^{-1}
\end{eqnarray*}
of the homomorphisms $\mathrm{can}, \varphi_p\colon TC^-_\ast(k) \to TP_\ast(k)$, where the variables $u$ and $\delta$ have degree $2$ and the variable $v$ has degree $-2$; see \cite[Props.~6.2-6.3]{BMS2}. 
\end{proof}
Thanks to Lemma \ref{lem:factorization}, we have a $K$-linear homomorphism:
\begin{equation}\label{eq:cycle-class}
K_0(\cA)_K \too TP_0(\cA)_{1/p}^{\varphi^n}\,.
\end{equation}
The $p$-version of Tate's conjecture admits the following noncommutative analogue:

\vspace{0.2cm}

Conjecture $T^p_{\mathrm{nc}}(\cA)$: {\it The homomorphism \eqref{eq:cycle-class} is surjective}.
\subsection*{Noncommutative Beilinson conjecture}
Recall from \cite[\S4.7]{book} that the group $K_0(\cA):=K_0(\cD_c(\cA))$ comes equipped with the Euler bilinear pairing:
\begin{eqnarray*}
\chi\colon K_0(\cA) \times K_0(\cA) \too \bbZ && ([M],[N]) \mapsto \sum_i (-1)^i \mathrm{dim}\, \Hom_{\cD_c(\cA)}(M,N[-i])\,.
\end{eqnarray*}
This bilinear pairing is, in general, not symmetric neither skew-symmetric. Nevertheless, as proved in \cite[Prop.~4.24]{book}, the left and right kernels agree. Consequently, we obtain the {\em numerical Grothendieck group} $K_0(\cA)/_{\!\sim\mathrm{num}}:=K_0(\cA)/\mathrm{Ker}(\chi)$.
\begin{notation}
Let $K_0(\cA)_\bbQ/_{\!\sim\mathrm{num}}:=K_0(\cA)_\bbQ/\mathrm{Ker}(\chi_\bbQ)\simeq (K_0(\cA)/_{\!\sim\mathrm{num}})_\bbQ$.
\end{notation}
Beilinson's conjecture admits the following noncommutative analogue:

\vspace{0.2cm}

Conjecture $B_{\mathrm{nc}}(\cA)$: {\it The equality $K_0(\cA)_\bbQ=K_0(\cA)_\bbQ/_{\!\sim \mathrm{num}}$ holds}.
\subsection*{Noncommutative Parshin conjecture}
Parshin's conjecture admits the following noncommutative analogue:

\vspace{0.2cm}

Conjecture $P_{\mathrm{nc}}(\cA)$: {\it We have $K_i(\cA)_\bbQ=0$ for every $i \geq1$}.
\section{Proof of Theorem \ref{thm:Thomason}}
As proved by Thomason in \cite{Thomason}, the Tate conjecture $T^l(X)$ is equivalent to the vanishing of the abelian groups $\Hom(\bbZ(l^\infty), \pi_{-1} L_{KU} K(X\times_{\bbF_q} \bbF_{q^m})), m \geq 1$. Therefore, the proof of the equivalence $T^l(X) \Leftrightarrow T^l_{\mathrm{nc}}(\perf_\dg(X))$ follows from the canonical Morita equivalence between the dg categories $\perf_\dg(X\times_{\bbF_q}\bbF_{q^m})$ and $\perf_\dg(X) \otimes_{\bbF_q} \bbF_{q^m}$; consult \cite[Lem.~4.26]{Gysin}.

Let us now prove the equivalence $T^p(X) \Leftrightarrow T^p_{\mathrm{nc}}(\perf_\dg(X))$. Recall that the ring of $p$-typical Witt vectors $W(k)$ is the unramified extension of degree $m$ of the ring of $p$-adic integers $\bbZ_p$. Hence, we have an induced field extension $\bbQ_p \to K$. Note that the cycle class map \eqref{eq:cycle-crystalline} is surjective if and only if the $K$-linear homomorphism 
$$\cZ^\ast(X)\otimes_{\bbQ_p}K \too H^{2\ast}_{\mathrm{crys}}(X)(\ast)^{\mathrm{Fr}_p}\otimes_{\bbQ_p}K$$ is surjective. Therefore, making use of the following natural isomorphisms
$$ H^{2\ast}_{\mathrm{crys}}(X)(\ast)^{\mathrm{Fr}_p}\otimes_{\bbQ_p}K \simeq H^{2\ast}_{\mathrm{crys}}(X)^{\frac{1}{p^\ast} \mathrm{Fr}_p}\otimes_{\bbQ_p}K \simeq H^{2\ast}_{\mathrm{crys}}(X)^{\frac{1}{(p^\ast)^n}\mathrm{Fr}_p^n}=H^{2\ast}_{\mathrm{crys}}(X)^{\frac{1}{q^\ast}\mathrm{Fr}_q}\,,$$
we conclude that the $p$-version of the Tate conjecture $T^p(X)$ is equivalent to the surjectivity of the induced $K$-linear cycle class map
\begin{equation}\label{eq:diagonal}
\cZ^\ast(X)_K \too H^{2\ast}_{\mathrm{crys}}(X)^{\frac{1}{q^\ast}\mathrm{Fr}_q}\,.
\end{equation}
On the one hand, since $\mathrm{char}(K)=0$, recall from \cite[\S18.3]{Fulton} that we have a natural isomorphism $K_0(\perf_\dg(X))_K \simeq \cZ^\ast(X)_K$. On the other hand, recall from \cite[Thm.~5.2]{CD-positive} that we have a natural isomorphism $TP_0(\perf_\dg(X))_{1/p} \simeq H^{2\ast}_{\mathrm{crys}}(X)$. Under these isomorphisms, the endomorphism $\varphi^n$ corresponds to the endomorphism $\frac{1}{q^\ast} \mathrm{Fr}_q$ (see \cite{Lars}) and the homomorphism \eqref{eq:diagonal} corresponds to the $K$-linear homomorphism \eqref{eq:cycle-class}. Consequently, \eqref{eq:diagonal} is surjective if and only if \eqref{eq:cycle-class} is surjective.

Let us now prove the equivalence $B(X) \Leftrightarrow B_{\mathrm{nc}}(\perf_\dg(X))$. Note first that since $\cD_c(\perf_\dg(X))\simeq \perf(X)$, the Euler bilinear pairing is given as follows:
\begin{eqnarray*}
\chi\colon K_0(X) \times K_0(X) \too \bbZ && ([\cF],[\cG]) \mapsto \sum_i (-1)^i \mathrm{dim}\, \Hom_{\perf(X)}(\cF,\cG[-i])\,.
\end{eqnarray*}
Recall from \cite[\S19]{Fulton} that an algebraic cycle $\beta\in \cZ^\ast(X)_\bbQ$ is {\em numerically equivalent to zero} if $\int_X \alpha \cdot \beta=0$ for every $\alpha \in \cZ^\ast(X)_\bbQ$. Recall also that we have the isomorphism
\begin{eqnarray}\label{eq:iso-key}
\tau\colon K_0(X)_\bbQ \stackrel{\sim}{\too} \cZ^\ast(X)_\bbQ && [\cF] \mapsto \mathrm{ch}(\cF)\cdot\sqrt{\mathrm{Td}}_X\,,
\end{eqnarray}
where $\mathrm{ch}(\cF)$ stands for the Chern character of $\cF$ and $\sqrt{\mathrm{Td}}_X$ for the square root of the Todd class of $X$; see \cite[\S18.3]{Fulton}. Given any two perfect complexes $\cF, \cG \in \perf(X)$, the Hirzebruch-Riemann-Roch theorem (see \cite[Cor.~18.3.1]{Fulton}) yields the equality
\begin{equation}\label{eq:equality1}
\mathrm{Eu}(\pi_\ast(\cF^\vee \otimes_{\cO_X} \cG))=\int_X \tau([\cF^\vee])\cdot \tau([\cG])\,,
\end{equation}
where $\mathrm{Eu}$ denotes the Euler characteristic and $\pi\colon X \to \mathrm{Spec}(k)$ denotes the structural morphism of $X$. Since $\cF^\vee \otimes_{\cO_X} \cG\simeq \uHom(\cF,\cG)$, where $\uHom(-,-)$ stands for the internal Hom of the rigid symmetric monoidal category $\perf(X)$, we hence conclude that $\mathrm{Eu}(\pi_\ast (\uHom(\cF,\cG)))$=\eqref{eq:equality1} agrees with $\chi([\cF],[\cG])$. This implies that the above isomorphism \eqref{eq:iso-key} descends to the numerical quotients:
\begin{equation}\label{eq:square}
\xymatrix{
K_0(X)_\bbQ \ar@{->>}[d] \ar[r]^-{\tau}_-{\sim} & \cZ^\ast(X)_\bbQ \ar@{->>}[d] \\
K_0(X)_\bbQ/_{\!\sim \mathrm{num}} \ar[r]_-{\tau}^-{\sim} & \cZ^\ast(X)_\bbQ/_{\!\sim\mathrm{num}}\,.
}
\end{equation}
Consequently, the proof of the equivalence $B(X) \Leftrightarrow B_{\mathrm{nc}}(\perf_\dg(X))$ follows now from the fact that $B(X)$, resp. $B_{\mathrm{nc}}(\perf_\dg(X))$, is equivalent to the injectivity of the vertical homomorphism on the right-hand side, resp. left hand-side, of \eqref{eq:square}.

Finally, the proof of the equivalence $P(X) \Leftrightarrow P_{\mathrm{nc}}(\perf_\dg(X))$ is clear.
\section{Proof of Theorem \ref{thm:main}}
By definition of the Lefschetz decomposition $\langle \bbA_0, \bbA_1(1), \ldots, \bbA_{i-1}(i-1)\rangle$, we have a chain of admissible triangulated subcategories $\bbA_{i-1}\subseteq \cdots \subseteq \bbA_1\subseteq \bbA_0$ with $\bbA_r(r):=\bbA_r \otimes \cL_X(r)$. Note that $\bbA_r(r)\simeq \bbA_r$. Let $\mathfrak{a}_r$ be the right orthogonal complement to $\bbA_{r+1}$ in $\bbA_r$; these are called the {\em primitive subcategories} in \cite[\S4]{Kuznetsov-IHES}. By construction, we have the following semi-orthogonal decompositions:
\begin{eqnarray}\label{eq:decomp1}
\bbA_r = \langle \mathfrak{a}_r, \mathfrak{a}_{r+1}, \ldots, \mathfrak{a}_{i-1} \rangle && 0\leq r \leq i-1\,.
\end{eqnarray}
As proved in \cite[Thm.~6.3]{Kuznetsov-IHES} (see also \cite[Thm.~2.3.4]{Bernardara}), the category $\perf(Y)$ admits a HP-dual Lefschetz decomposition $\langle \bbB_{j-1}(1-j), \bbB_{j-2}(2-j), \ldots, \bbB_0\rangle$ with respect to $\cL_Y(1)$; as above, we have a chain of admissible triangulated subcategories $\bbB_{j-1} \subseteq \bbB_{j-2} \subseteq \cdots \subseteq \bbB_0$. Moreover, the primitive subcategories coincide (via a Fourier-Mukai type functor) with those of $\perf(X)$ and we have the following semi-orthogonal decompositions:
\begin{eqnarray}\label{eq:decomp2}
\bbB_r=\langle \mathfrak{a}_0, \mathfrak{a}_1, \ldots, \mathfrak{a}_{\mathrm{dim}(V)-r-2}\rangle && 0 \leq r \leq j-1\,.
\end{eqnarray}
Furthermore, the assumptions $\mathrm{dim}(X_L)=\mathrm{dim}(X)-\mathrm{dim}(L)$ and $\mathrm{dim}(Y_L)=\mathrm{dim}(Y) - \mathrm{dim}(L^\perp)$ imply the existence of semi-orthogonal decompositions
\begin{equation}\label{eq:semi-1}
\perf(X_L)=\langle \bbC_L, \bbA_{\mathrm{dim}(V)}(1), \ldots, \bbA_{i-1}(i-\mathrm{dim}(V))\rangle 
\end{equation}
\begin{equation}\label{eq:semi-2}
\perf(Y_L)=\langle \bbB_{j-1}(\mathrm{dim}(L^\perp)-j), \ldots, \bbB_{\mathrm{dim}(L^\perp)}(-1), \bbC_L \rangle\,,
\end{equation}
where $\bbC_L$ is a common (triangulated) category. Let us denote, respectively, by $\bbC_L^\dg$, $\bbA_r^\dg$, and $\mathfrak{a}_r^\dg$, the dg enhancement of $\bbC_L$, $\bbA_r$, and $\mathfrak{a}_r$, induced from $\perf_\dg(X_L)$. Similarly, let us denote by $\bbC_L^{\dg'}$ and $\bbB_r^{\dg}$ the dg enhancement of $\bbC_L$ and $\bbB_r$ induced from $\perf_\dg(Y_L)$. Note that since by assumption the $k$-schemes $X_L$ and $Y_L$ are smooth (and projective), all the above dg categories are smooth (and proper). 

Let us now prove the equivalence $T^l(X_L)\Leftrightarrow T^l(Y_L)$. Consider the functors 
\begin{eqnarray}\label{eq:additive}
E_m\colon \dgcat(k) \too \mathrm{Mod}(\bbZ) && \cA \mapsto \pi_{-1} L_{KU} K(\cA\otimes_{\bbF_q} \bbF_{q^m})
\end{eqnarray}
with values in the additive category of abelian groups. 
\begin{proposition}\label{prop:additive}
The functors \eqref{eq:additive} are additive invariants.
\end{proposition}
\begin{proof}
Let $F\colon \cA \to \cB$ be a Morita equivalence and $m\geq 1$ an integer. As proved in \cite[Prop.~7.1]{Artin}, the induced dg functor $F\otimes_{\bbF_q} \bbF_{q^m} \colon \cA\otimes_{\bbF_q} \bbF_{q^m} \to \cB \otimes_{\bbF_q} \bbF_{q^m}$ is also a Morita equivalence. Therefore, since algebraic $K$-theory inverts Morita equivalences (see \cite[\S2.2.1]{book}), the homomorphism $K(F\otimes_{\bbF_q} \bbF_{q^m}) \colon K(\cA\otimes_{\bbF_q} \bbF_{q^m}) \to K(\cB \otimes_{\bbF_q} \bbF_{q^m})$ is invertible. By definition of the above functors \eqref{eq:additive}, we hence conclude that the induced group homomorphism $E_m(\cA) \to E_m(\cB)$ is also invertible.

Now, let $\cA$ and $\cB$ be two dg categories and $\mathrm{B}$ a dg $\cA\text{-}\cB$-bimodule. Following \S\ref{sub:additive}, we need to show that the dg functors $\iota_\cA$ and $\iota_B$ induce an isomorphism
\begin{equation}\label{eq:iso-induced}
E_m(\cA) \oplus E_m(\cB) \too E_m(T(\cA,\cB;\mathrm{B}))\,.
\end{equation}
Consider the dg categories $\cA\otimes_{\bbF_q} \bbF_{q^m}$ and $\cB \otimes_{\bbF_q} \bbF_{q^m}$ and the dg bimodule $\mathrm{B}\otimes_{\bbF_q} \bbF_{q^m}$. Since algebraic $K$-theory is an additive invariant of dg categories, the dg functors $\iota_{\cA\otimes_{\bbF_q} \bbF_{q^m}}$ and $\iota_{\cB \otimes_{\bbF_q} \bbF_{q^m}}$ induce an isomorphism
\begin{equation}\label{eq:isom-last}
K(\cA\otimes_{\bbF_q} \bbF_{q^m})\oplus K(\cB \otimes_{\bbF_q} \bbF_{q^m}) \stackrel{\sim}{\too} K(T(\cA\otimes_{\bbF_q} \bbF_{q^m}, \cB\otimes_{\bbF_q} \bbF_{q^m}; \mathrm{B}\otimes_{\bbF_q} \bbF_{q^m}))\,.
\end{equation}
Therefore, by definition of the above functors \eqref{eq:additive}, we conclude from \eqref{eq:isom-last} that the homomorphism \eqref{eq:iso-induced} is also invertible. This finishes the proof.
\end{proof}
Thanks to Proposition \ref{prop:additive}, the functors \eqref{eq:additive} are additive invariants. As explained in \cite[Prop.~2.2]{book}, this implies that the above semi-orthogonal decompositions \eqref{eq:semi-1}-\eqref{eq:semi-2} give rise to direct sum decompositions of abelian groups:
\begin{equation}\label{eq:group1} 
E_m(\perf_\dg(X_L)) \simeq E_m(\bbC^\dg_L) \oplus E_m(\bbA^\dg_{\mathrm{dim}(V)}) \oplus \cdots \oplus E_m(\bbA^\dg_{i-1})\end{equation}
\begin{equation}\label{eq:group2}
E_m(\perf_\dg(Y_L)) \simeq E_m(\bbB_{j-1}^\dg) \oplus \cdots \oplus E_m(\bbB^\dg_{\mathrm{dim}(L^\perp)}) \oplus E_m(\bbC_L^{\dg'})\,.
\end{equation}
Consequently, by applying the functor $\mathrm{Hom}(\bbZ(l^\infty), -)$ to the direct sum decompositions \eqref{eq:group1}-\eqref{eq:group2}, we obtain the following equivalences of conjectures:
\begin{equation}\label{eq:conjecture1}
T^l_{\mathrm{nc}}(\perf_\dg(X_L))\Leftrightarrow T^l_{\mathrm{nc}}(\bbC_L^\dg)+ T^l_{\mathrm{nc}}(\bbA^\dg_{\mathrm{dim}(V)}) + \cdots + T^l_{\mathrm{nc}}(\bbA^\dg_{i-1})
\end{equation}
\begin{equation}\label{eq:conjecture2}
T^l_{\mathrm{nc}}(\perf_\dg(Y_L))\Leftrightarrow T^l_{\mathrm{nc}}(\bbB^\dg_{j-1}) + \cdots + T^l_{\mathrm{nc}}(\bbB^\dg_{\mathrm{dim}(L^\perp)}) + T^l_{\mathrm{nc}}(\bbC^{\dg'}_L)\,.
\end{equation}
On the one hand, since by assumption the conjecture $T^l_{\mathrm{nc}}(\bbA_0^\dg)$ holds, we conclude from the above semi-orthogonal decompositions \eqref{eq:decomp1}-\eqref{eq:decomp2} that the conjectures $T^l_{\mathrm{nc}}(\bbA^\dg_r)$ and $T^l_{\mathrm{nc}}(\bbB^\dg_r)$, with $0 \leq r \leq i-1$, also hold. This implies that the right-hand side of \eqref{eq:conjecture1}, resp. \eqref{eq:conjecture2}, reduces to the conjecture $T^l_{\mathrm{nc}}(\bbC^\dg_L)$, resp. $T^l_{\mathrm{nc}}(\bbC_L^{\dg'})$. On the other hand, since the functor $\perf(X_L) \to \bbC_L \to \perf(Y_L)$ is of Fourier-Mukai type, the dg categories $\bbC_L^\dg$ and $\bbC_L^{\dg'}$ are Morita equivalent. Using the fact that the functors \eqref{eq:additive} invert Morita equivalences, we hence conclude that $T^l_{\mathrm{nc}}(\bbC^\dg_L)\Leftrightarrow T^l_{\mathrm{nc}}(\bbC_L^{\dg'})$. Consequently, the proof follows now from the equivalences $T^l(X_L) \Leftrightarrow T^l_{\mathrm{nc}}(\perf_\dg(X_L))$ and $T^l(Y_L) \Leftrightarrow T^l_{\mathrm{nc}}(\perf_\dg(Y_L))$ of Theorem \ref{thm:main}.

Let us now prove the equivalence $T^p(X_L)\Leftrightarrow T^p(Y_L)$. As explained in \cite[Prop.~2.2]{book}, since the functor \eqref{eq:universal} is an additive invariant, the above semi-orthogonal decomposition \eqref{eq:semi-1} gives rise to the following direct sum decomposition
\begin{equation}\label{eq:decomp-1}
U(\perf_\dg(X_L))_K \simeq U(\bbC^\dg_L)_K \oplus U(\bbA^\dg_{\mathrm{dim}(V)})_K \oplus \cdots \oplus U(\bbA^\dg_{i-1})_K
\end{equation}
in the $K$-linearized category $\Hmo_0(k)_K$. Recall from the proof of Lemma \ref{lem:factorization} that the functor \eqref{eq:TP} comes equipped with the natural transformation $\varphi^n$. Therefore, by applying the $K$-linear functor \eqref{eq:TP} to the above direct sum decomposition \eqref{eq:decomp-1}, we conclude that the induced $K$-linear homomorphism 
$$K_0(\perf_\dg(X_L))_K \too TP_0(\perf_\dg(X_L))_{1/p}^{\varphi^n}$$
identifies with the induced (diagonal) $K$-linear homomorphism
$$ K_0(\bbC_L^\dg) \oplus \oplus_{r=\mathrm{dim}(V)}^{i-1} K_0(\bbA^\dg_r)_K \too TP_0(\bbC_L^\dg)^{\varphi^n}_{1/p} \oplus \oplus_{r=\mathrm{dim}(V)}^{i-1} TP_0(\bbA^\dg_r)^{\varphi^n}_{1/p}\,.$$
This implies the following equivalence of conjectures:
$$ T^p_{\mathrm{nc}}(\perf_\dg(X_L))\Leftrightarrow T^p_{\mathrm{nc}}(\bbC_L^\dg)+ T^p_{\mathrm{nc}}(\bbA^\dg_{\mathrm{dim}(V)}) + \cdots + T^p_{\mathrm{nc}}(\bbA^\dg_{i-1})\,.$$
All the above holds {\em mutatis mutandis} with $X_L$ replaced by $Y_L$. Consequently, the above semi-orthogonal decomposition \eqref{eq:semi-2} leads to the equivalence of conjectures
$$T^p_{\mathrm{nc}}(\perf_\dg(Y_L))\Leftrightarrow T^p_{\mathrm{nc}}(\bbB^\dg_{j-1}) + \cdots + T^p_{\mathrm{nc}}(\bbB^\dg_{\mathrm{dim}(L^\perp)}) + T^p_{\mathrm{nc}}(\bbC^{\dg'}_L)\,.$$
The remainder of the proof is now similar to the proof of $T^l(X_L)\Leftrightarrow T^l(Y_L)$.

Let us now prove the equivalence $B(X_L)\Leftrightarrow B(Y_L)$. As above, the semi-orthogonal decompositions \eqref{eq:semi-1}-\eqref{eq:semi-2} give rise to the following direct sum decompositions
\begin{equation}\label{eq:decomp11}
U(\perf_\dg(X_L))_\bbQ \simeq U(\bbC^\dg_L)_\bbQ \oplus U(\bbA^\dg_{\mathrm{dim}(V)})_\bbQ \oplus \cdots \oplus U(\bbA^\dg_{i-1})_\bbQ
\end{equation}
\begin{equation}\label{eq:decomp22}
U(\perf_\dg(Y_L))_\bbQ \simeq U(\bbB_{j-1}^\dg)_\bbQ \oplus \cdots \oplus U(\bbB^\dg_{\mathrm{dim}(L^\perp)})_\bbQ \oplus U(\bbC_L^{\dg'})_\bbQ
\end{equation}
in the $\bbQ$-linearized category $\Hmo_0(k)_\bbQ$. As proved in \cite[\S6]{positive}, given any smooth proper dg category $\cA$, we have a natural isomorphism: 
\begin{equation}\label{eq:numerical}
\Hom_{\NNum(k)_\bbQ}(U(k)_\bbQ, U(\cA)_\bbQ)\simeq K_0(\cA)_\bbQ/_{\!\sim \mathrm{num}}\,.
\end{equation}
Hence, by applying $\Hom_{\NChow(k)_\bbQ}(U(k)_\bbQ,-)$ and $\Hom_{\NNum(k)_\bbQ}(U(k)_\bbQ,-)$ to the direct sum decompositions \eqref{eq:decomp11}-\eqref{eq:decomp22} we obtain the equivalences of conjectures:
\begin{equation*}
B_{\mathrm{nc}}(\perf_\dg(X_L))\Leftrightarrow B_{\mathrm{nc}}(\bbC_L^\dg)+ B_{\mathrm{nc}}(\bbA^\dg_{\mathrm{dim}(V)}) + \cdots + B_{\mathrm{nc}}(\bbA^\dg_{i-1})
\end{equation*}
\begin{equation*}
B_{\mathrm{nc}}(\perf_\dg(Y_L))\Leftrightarrow B_{\mathrm{nc}}(\bbB^\dg_{j-1}) + \cdots + B_{\mathrm{nc}}(\bbB^\dg_{\mathrm{dim}(L^\perp)}) + B_{\mathrm{nc}}(\bbC^{\dg'}_L)\,.
\end{equation*}
The remainder of the proof is now similar to the proof of $T^l(X_L)\Leftrightarrow T^l(Y_L)$.

Finally, let us prove the equivalence $P(X_L) \Leftrightarrow P(Y_L)$. Consider the functors
\begin{eqnarray}\label{eq:K-theory}
K_i(-)_\bbQ\colon \dgcat(k) \too \mathrm{Vect}(\bbQ) && \cA \mapsto K_i(\cA)_\bbQ
\end{eqnarray}
with values in the category of $\bbQ$-vector spaces. As explained in \cite[\S2.2.1]{book}, these functors are additive invariants. Therefore, a proof similar to the one of the equivalence $T^l(X_L)\Leftrightarrow T^l(Y_L)$ allows us to conclude that $P(X_L) \Leftrightarrow P(Y_L)$.
\section{Proof of Proposition \ref{prop:key}}
As proved in \cite[Thms.~1.3 and 1.7]{VdB}, the dg category $\perf_\dg(\mathrm{Gr}(r,U_1))$ is Morita equivalent to a finite dimensional $k$-algebra of finite global dimension $A$. Since $\bbA_0^\dg=\perf_\dg(\mathrm{Gr}(r,U_1))$, we hence obtain the following equivalences of conjectures:
\begin{eqnarray*}
T^l_{\mathrm{nc}}(\bbA_0^\dg) \Leftrightarrow T^l_{\mathrm{nc}}(A) && T^p_{\mathrm{nc}}(\bbA_0^\dg) \Leftrightarrow T^p_{\mathrm{nc}}(A) \\ 
B_{\mathrm{nc}}(\bbA_0^\dg) \Leftrightarrow B_{\mathrm{nc}}(A) && P_{\mathrm{nc}}(\bbA_0^\dg) \Leftrightarrow P_{\mathrm{nc}}(A)\,.
\end{eqnarray*}
We start by proving the conjecture $T^l_{\mathrm{nc}}(A)$. Recall that every finite field $k$ is perfect. Therefore, using the fact that the above functors \eqref{eq:additive} are additive invariants, \cite[Thm.~3.15]{Azumaya} implies that $E_m(A)\simeq E_m(A/J(A)), m\geq 1$, where $J(A)$ stands for the Jacobson radical of $A$. By applying the functor $\Hom(\bbZ(l^\infty),-)$ to these latter isomorphisms, we hence conclude that $T^l_{\mathrm{nc}}(A)\Leftrightarrow T^l_{\mathrm{nc}}(A/J(A))$. Now, let us write $V_1, \ldots, V_s$ for the simple (right) $A/J(A)$-modules and $D_1:=\mathrm{End}_{A/J(A)}(V_1), \ldots, D_s:=\mathrm{End}_{A/J(A)}(V_s)$ the associated division $k$-algebras. Thanks to the Artin-Wedderburn theorem, the quotient $A/J(A)$ is Morita equivalent to $D_1 \times \cdots \times D_s$. Moreover, the center $Z_i$ of $D_i$ is a finite field extension of $k$ and $D_i$ is a central simple $Z_i$-algebra. Since the Brauer group of a finite field $k$ is trivial, this implies that $D_1 \times \cdots \times D_s$ is Morita equivalent to $Z_1 \times \cdots \times Z_s$. Consequently, we obtain the following equivalences:
$$
T^l_{\mathrm{nc}}(A) \Leftrightarrow T^l_{\mathrm{nc}}(A/J(A)) \Leftrightarrow T^l_{\mathrm{nc}}(D_1) + \cdots + T^l_{\mathrm{nc}}(D_s) \Leftrightarrow T^l(Z_1) + \cdots + T^l(Z_s)\,.
$$
The proof of the conjecture $T^l_{\mathrm{nc}}(A)$ follows now from the fact that the conjectures $T^l(Z_i), 1\leq i \leq s$, hold because $\mathrm{dim}(Z_i)=0$.

Let us now prove the conjecture $T^p_{\mathrm{nc}}(A)$. Since the functor \eqref{eq:universal} is an additive invariant, \cite[Thm.~3.15]{Azumaya} implies that $U(A)_K\simeq U(A/J(A))_K$ in $\Hmo_0(k)_K$. Recall from the proof of Lemma \ref{lem:factorization} that the functor \eqref{eq:TP} comes equipped with the natural transformation $\varphi^n$. Therefore, by applying the $K$-linear functor \eqref{eq:TP} to the latter isomorphism, we obtain the following identification:
$$ \left(K_0(A)_K \too TP_0(A)^{\varphi^n}_{1/p}\right) \simeq \left(K_0(A/J(A))_K \too TP_0(A/J(A))^{\varphi^n}_{1/p}\right)\,.$$
This implies the equivalence of conjectures $T^p_{\mathrm{nc}}(A) \Leftrightarrow T^p_{\mathrm{nc}}(A/J(A))$. The proof of the conjecture $T^p_{\mathrm{nc}}(A/J(A))$ is now similar to the proof of $T^l_{\mathrm{nc}}(A/J(A))$. 

Let us now prove the conjecture $B_{\mathrm{nc}}(A)$. As above, we have an isomorphism $U(A)_\bbQ\simeq U(A/J(A))_\bbQ$ in $\Hmo_0(k)_\bbQ$. Thanks to the natural isomorphisms \eqref{eq:isos} and \eqref{eq:numerical}, by applying $\Hom_{\NChow(k)_\bbQ}(U(k)_\bbQ, -)$ and $\Hom_{\NNum(k)_\bbQ}(U(k)_\bbQ, -)$ to the latter isomorphism, we hence conclude that $B_{\mathrm{nc}}(A)\Leftrightarrow B_{\mathrm{nc}}(A/J(A))$. The proof of the conjecture $B_{\mathrm{nc}}(A/J(A))$ is now similar to the proof of $T^l_{\mathrm{nc}}(A/J(A))$.

Finally, since the above functors \eqref{eq:K-theory} are additive invariants, the proof of the conjecture $P_{\mathrm{nc}}(A)$ is similar to the proof of $T^l_{\mathrm{nc}}(A)$.
\section{Proof of Theorem \ref{thm:two}}\label{sec:proof}
We assume first that $d$ is even. Following \cite[\S3.5]{Kuznetsov-quadrics} (see also \cite[\S1.6]{Bernardara}), let $\cZ$ be the center of $\mathcal{C}l_0(q)_{|L}$ and $\mathrm{Spec}(\cZ)=:\widetilde{\bbP}(L) \to \bbP(L)$ the {\em discriminant cover} of $\bbP(L)$. As explained in {\em loc. cit.}, $\widetilde{\bbP}(L) \to \bbP(L)$ is a $2$-fold cover which is ramified over the divisor $D:=\bbP(L) \cap \Delta_1$. Since by assumption $\mathrm{dim}(L)=2$, we have $\mathrm{dim}(D)=0$. Consequently, since $D$ is smooth, $\widetilde{\bbP}(L)$ is also smooth. Let us write $\cF$ for the sheaf of noncommutative algebras $\mathcal{C}l_0(q)_{|L}$ considered as a sheaf of noncommutative algebras over $\widetilde{\bbP}(L)$. As proved in {\em loc. cit.}, since by assumption $\bbP(L)\cap \Delta_2=\emptyset$, $\cF$ is a sheaf of Azumaya algebras over $\widetilde{\bbP}(L)$ of rank $2^{(d/2)-1}$. Moreover, the category $\perf(\bbP(L); \mathcal{C}l_0(q)_{|L})$ is equivalent (via a Fourier-Mukai type functor) to $\perf(\widetilde{\bbP}(L); \cF)$. This leads to a Morita equivalence between the dg categories $\perf_\dg(\bbP(L); \mathcal{C}l_0(q)_{|L})$ and $\perf_\dg(\widetilde{\bbP}(L); \cF)$. Consequently, making use of Corollary \ref{cor:main}, we obtain the following equivalences of conjectures:
\begin{eqnarray}
& T^l(X_L) \Leftrightarrow T^l_{\mathrm{nc}}(\perf_\dg(\widetilde{\bbP}(L); \cF)) & T^p(X_L) \Leftrightarrow T^p_{\mathrm{nc}}(\perf_\dg(\widetilde{\bbP}(L); \cF))  \label{eq:equivalence-new}\\
& B(X_L)  \Leftrightarrow B_{\mathrm{nc}}(\perf_\dg(\widetilde{\bbP}(L); \cF)) & P(X_L) \Leftrightarrow P_{\mathrm{nc}}(\perf_\dg(\widetilde{\bbP}(L); \cF)) \label{eq:equivalence-new111}\,.
\end{eqnarray}
Since by assumption $\mathrm{dim}(L)=2$, the $2$-fold cover $\widetilde{\bbP}(L)$ is a smooth projective curve. Using the fact that the Brauer group of every smooth curve over a finite field $k$ is trivial (see \cite[page 109]{Milne}), we hence conclude that the right-hand side conjectures in \eqref{eq:equivalence-new}-\eqref{eq:equivalence-new111} are equivalent to $T^l(\widetilde{\bbP}(L))$, $T^p(\widetilde{\bbP}(L))$, $B(\widetilde{\bbP}(L))$, and $P(\widetilde{\bbP}(L))$, respectively. The proof follows now from the fact that the Tate, Beilinson and Parshin conjectures hold for smooth projective curves. 

We now assume that $d$ is odd and that $p\neq 2$. Following \cite[\S3.6]{Kuznetsov-quadrics} (see also \cite[\S1.7]{Bernardara}), let $\widehat{\bbP}(L)$ be the {\em discriminant stack} associated to the pull-back $q_{|L}$ along $\bbP(L) \subset \bbP(S^2W^\ast)$ of the flat quadric fibration $q\colon \cH \to \bbP(S^2W^\ast)$. As explained in {\em loc. cit.}, since by assumption $1/2 \in k$, $\widehat{\bbP}(L)$ is a smooth Deligne-Mumford stack. Moreover, using the fact that $\widehat{\bbP}(L)$ is a square root stack and that the critical locus of the flat quadric fibration $q_{|L}$ is the divisor $D$, we conclude from \cite[Thm.~1.6]{Ueda} that $\perf(\widehat{\bbP}(L))=\langle \perf(D), \perf(\bbP(L))\rangle$. Consequently, an argument similar to the one used in the proof of Theorem \ref{thm:main} yields the following equivalences of conjectures:
\begin{eqnarray}
T^l_{\mathrm{nc}}(\perf_\dg(\widehat{\bbP}(L))) & \Leftrightarrow & T^l(D) + T^l(\bbP(L)) \label{eq:aux1} \\
T^p_{\mathrm{nc}}(\perf_\dg(\widehat{\bbP}(L))) & \Leftrightarrow & T^p(D) + T^p(\bbP(L)) \label{eq:aux11} \\
B_{\mathrm{nc}}(\perf_\dg(\widehat{\bbP}(L))) & \Leftrightarrow & B(D) + B(\bbP(L))  \label{eq:aux2} \\
P_{\mathrm{nc}}(\perf_\dg(\widehat{\bbP}(L))) & \Leftrightarrow & P(D) + P(\bbP(L))\,. \label{eq:aux3}
\end{eqnarray} 
Let us write $\cF$ for the sheaf of noncommutative algebras $\mathcal{C}l_0(q)_{|L}$ considered as a sheaf of noncommutative algebras over $\widehat{\bbP}(L)$. As proved in \cite[\S3.6]{Kuznetsov-quadrics} (see also \cite[\S1.7]{Bernardara}), since by assumption $\bbP(L) \cap \Delta_2=\emptyset$, $\cF$ is a sheaf of Azumaya algebras over $\widehat{\bbP}(L)$. Moreover, the category $\perf(\bbP(L); \mathcal{C}l_0(q)_{|L})$ is equivalent (via a Fourier-Mukai type functor) to $\perf(\widehat{\bbP}(L); \cF)$. This leads to a Morita equivalence between the dg categories $\perf_\dg(\bbP(L); \mathcal{C}l_0(q)_{|L})$ and $\perf_\dg(\widehat{\bbP}(L); \cF)$. Making use of Corollary \ref{cor:main}, we hence obtain the following equivalences of conjectures:
\begin{eqnarray}
& T^l(X_L)  \Leftrightarrow   T^l_{\mathrm{nc}}(\perf_\dg(\widehat{\bbP}(L); \cF)) & T^p(X_L) \Leftrightarrow  T^p_{\mathrm{nc}}(\perf_\dg(\widehat{\bbP}(L); \cF))\label{eq:equivalence-new1}\\
& B(X_L)  \Leftrightarrow  B_{\mathrm{nc}}(\perf_\dg(\widehat{\bbP}(L); \cF)) & P(X_L) \Leftrightarrow  P_{\mathrm{nc}}(\perf_\dg(\widehat{\bbP}(L); \cF)) \label{eq:equivalence-new11}\,.
\end{eqnarray}
Since by assumption $\mathrm{dim}(L)=2$, we have $\mathrm{dim}(\bbP(L))=1$. Using the fact that the Brauer group of every smooth curve over a finite field $k$ is trivial, we hence conclude that in \eqref{eq:equivalence-new1}-\eqref{eq:equivalence-new11} we can replace $\perf_\dg(\widehat{\bbP}(L); \cF)$ by $\perf_\dg(\widehat{\bbP}(L))$. Consequently, since $\mathrm{dim}(D)=0$, the proof follows now from the combination of \eqref{eq:aux1}-\eqref{eq:aux3} with the fact that the Tate, Beilinson and Parshin conjectures hold in dimensions $\leq 1$.
\subsection*{Intersection of even-dimensional quadrics}
Given a smooth proper dg category $\cA$, a prime number $l\neq p$, and an integer $s \geq 1$, consider the $\bbZ[1/s]$-modules
\begin{eqnarray}\label{eq:groups2}
\Hom(\bbZ(l^\infty), (\pi_{-1} L_{KU}K(\cA\otimes_{\bbF_q} \bbF_{q^m}))_{1/s})&& m \geq 1
\end{eqnarray} 
as well as the following variant of the noncommutative Tate conjecture:

\vspace{0.2cm}

Conjecture $T^l_{\mathrm{nc}}(\cA; 1/s)$: {\it The $\bbZ[1/s]$-modules \eqref{eq:groups2} are trivial}.

\begin{lemma}\label{lem:aux}
We have $T^l_{\mathrm{nc}}(\cA) \Leftrightarrow T^l_{\mathrm{nc}}(\cA;1/s)$ for every $l \nmid s$.
\end{lemma}
\begin{proof}
Since by assumption $l \nmid s$, the localization homomorphisms
\begin{eqnarray*}
\pi_{-1} L_{KU} K(\cA\otimes_{\bbF_q} \bbF_{q^m}) \too (\pi_{-1} L_{KU} K(\cA\otimes_{\bbF_q} \bbF_{q^m}))_{1/s} && m \geq 1
\end{eqnarray*}
induce an isomorphism between all the $l$-power torsion subgroups. Consequently, by passing to the $l$-adic Tate modules, we conclude that the abelian groups \eqref{eq:groups11} are trivial if and only if the $\bbZ[1/s]$-modules \eqref{eq:groups2} are trivial.
\end{proof}
\begin{theorem}\label{prop:2}
Let $X_L$ be as in Corollary \ref{cor:main}. Assume that $\bbP(L) \cap \Delta_2 =\emptyset$, that the divisor $\bbP(L) \cap \Delta_1$ is smooth, and that $d$ is even. Under these assumptions, we have the following equivalences: 
\begin{eqnarray*}
& T^l(X_L) \Leftrightarrow T^l(\widetilde{\bbP}(L)) \,\,\, (\mathrm{for}\,\,\mathrm{every}\,\, l \neq 2) &\\
& T^p(X_L) \Leftrightarrow T^p(\widetilde{\bbP}(L)) \quad 
B(X_L) \Leftrightarrow B(\widetilde{\bbP}(L)) \quad P(X_L) \Leftrightarrow P(\widetilde{\bbP}(L))\,. &
\end{eqnarray*}
\end{theorem}
\begin{proof}
As explained in the proof of Theorem \ref{thm:two}, we have the equivalences
\begin{eqnarray}
& T^l(X_L) \Leftrightarrow T^l_{\mathrm{nc}}(\perf_\dg(\widetilde{\bbP}(L); \cF)) & T^p(X_L) \Leftrightarrow T^p_{\mathrm{nc}}(\perf_\dg(\widetilde{\bbP}(L); \cF)) \label{eq:eq:last111}\\
& B(X_L) \Leftrightarrow B_{\mathrm{nc}}(\perf_\dg(\widetilde{\bbP}(L); \cF)) & P(X_L) \Leftrightarrow P_{\mathrm{nc}}(\perf_\dg(\widetilde{\bbP}(L); \cF)) \label{eq:eq:last1111}\,,
\end{eqnarray}
where $\cF$ is a certain sheaf of Azumaya algebras over $\widetilde{\bbP}(L)$ of rank $2^{(d/2)-1}$. 

We start by proving that the first right-hand side conjecture in \eqref{eq:eq:last111} is equivalent to $T^l(\widetilde{\bbP}(L))$ (for every $l\neq 2$). Consider the following functors 
\begin{eqnarray*}
& E_m(-)_{1/2}\colon \dgcat(k) \too \mathrm{Mod}(\bbZ[1/2]) & \cA \mapsto (\pi_{-1}L_{KU} K(\cA\otimes_{\bbF_q} \bbF_{q^m}))_{1/2}
\end{eqnarray*}
with values in the additive category of $\bbZ[1/2]$-modules. Similarly to the proof of Proposition \ref{prop:additive}, these functors are additive invariants. Consequently, using the fact that the rank of the sheaf of Azumaya algebras $\cF$  is a power of $2$ and that the category $\mathrm{Mod}(\bbZ[1/2])$ is $\bbZ[1/2]$-linear, we conclude from \cite[Thm.~2.1]{Azumaya} that $E_m(\perf_\dg(\widetilde{\bbP}(L)))_{1/2} \simeq E_m(\perf_\dg(\widetilde{\bbP}(L); \cF))_{1/2}$. By applying the functor $\Hom(\bbZ(l^\infty),-)$ to these isomorphisms, we hence obtain the equivalences:
\begin{equation}\label{eq:3}
T^l_{\mathrm{nc}}(\perf_\dg(\widetilde{\bbP}(L); \cF); 1/2) \Leftrightarrow T^l_{\mathrm{nc}}(\perf_\dg(\widetilde{\bbP}(L)); 1/2)) \Leftrightarrow T^l(\widetilde{\bbP}(L); 1/2)\,.
\end{equation}
Thanks to Lemma \ref{lem:aux}, the preceding equivalence \eqref{eq:3} yields the equivalence $T^l_{\mathrm{nc}}(\perf_\dg(\widetilde{\bbP}(L); \cF)) \Leftrightarrow T^l(\widetilde{\bbP}(L))$ (for every $l\neq 2$), and so the proof is finished.

Let us now prove that the second right-hand side conjecture in \eqref{eq:eq:last111} is equivalent to $T^p(\widetilde{\bbP}(L))$. Since the functor \eqref{eq:universal} is additive and $\mathrm{char}(K)=0$, \cite[Thm.~2.1]{Azumaya} implies that $U(\perf_\dg(\widetilde{\bbP}(L)))_K\simeq U(\perf_\dg(\widetilde{\bbP}(L);\cF))_K$ in $\Hmo_0(k)_K$. Recall from the proof of Lemma \ref{lem:factorization} comes equipped with the natural transformation $\varphi^n$. Therefore, by applying the $K$-linear functor \eqref{eq:TP} to the latter isomorphism, we conclude that the induced $K$-linear homomorphism
$$ K_0(\perf_\dg(\widetilde{\bbP}(L)))_K \too TP_0(\perf_\dg(\widetilde{\bbP}(L)))^{\varphi^n}_{1/p}$$
identifies with the induced $K$-linear homomorphism
$$ K_0(\perf_\dg(\widetilde{\bbP}(L); \cF))_K \too TP_0(\perf_\dg(\widetilde{\bbP}(L);\cF))^{\varphi^n}_{1/p}\,.$$
This implies the following equivalences of conjectures:
\begin{equation*}
T^p_{\mathrm{nc}}(\perf_\dg(\widetilde{\bbP}(L); \cF)) \Leftrightarrow T^p_{\mathrm{nc}}(\perf_\dg(\widetilde{\bbP}(L))) \Leftrightarrow T^p(\widetilde{\bbP}(L))\,.
\end{equation*}

Let us now prove that the first right-hand side conjecture in \eqref{eq:eq:last1111} is equivalent to $B(\widetilde{\bbP}(L))$. As above, we have $U(\perf_\dg(\widetilde{\bbP}(L))_\bbQ \simeq U(\perf_\dg(\widetilde{\bbP}(L); \cF))_\bbQ$ in $\Hmo_0(k)_\bbQ$. Thanks to the natural isomorphisms \eqref{eq:isos} and \eqref{eq:numerical}, by applying the functors $\Hom_{\NChow(k)_\bbQ}(U(k)_\bbQ, -)$ and $\Hom_{\NNum(k)_\bbQ}(U(k)_\bbQ, -)$ to the latter isomorphism, we hence obtain the following equivalences of conjectures:
\begin{equation*}\label{eq:4}
B_{\mathrm{nc}}(\perf_\dg(\widetilde{\bbP}(L); \cF)) \Leftrightarrow B_{\mathrm{nc}}(\perf_\dg(\widetilde{\bbP}(L))) \Leftrightarrow B(\widetilde{\bbP}(L))\,.
\end{equation*}
Finally, since the functors \eqref{eq:K-theory} are additive invariants (with values in the additive category of $\bbQ$-vector spaces), the proof that the second right-hand side conjecture in \eqref{eq:eq:last1111} is equivalent to $P(\widetilde{\bbP}(L))$ is similar to the proof that the first right-hand side conjecture in \eqref{eq:eq:last111} is equivalent to $T^l(\widetilde{\bbP}(L))$ (for every $l\neq 2$).
\end{proof}
\begin{remark}[Azumaya algebras]\label{rk:Azumaya}
Let $X$ be a smooth projective $k$-scheme and $\cF$ a sheaf of Azumaya algebras over $X$ of rank $r$. Note that an argument similar to the one used in the proof of Theorem \ref{prop:2} leads to the following equivalences:
\begin{eqnarray*}
& \,\,\,\,T^l_{\mathrm{nc}}(\perf_\dg(X;\cF)) \Leftrightarrow T^l(X) & (\mathrm{for}\,\,\mathrm{every}\,\, l \neq r) \\
& \,\,\,\,\,T^p_{\mathrm{nc}}(\perf_\dg(X;\cF)) \Leftrightarrow T^p(X) & \\
& B_{\mathrm{nc}}(\perf_\dg(X;\cF))) \Leftrightarrow B(X) & \\
& \,\,\,\, P_{\mathrm{nc}}(\perf_\dg(X;\cF))) \Leftrightarrow P(X)\,. &
\end{eqnarray*}
\end{remark}

\section{Proof of Theorem \ref{thm:orbifold}}
We start by proving Theorem \ref{thm:orbifold} in what regards the Tate conjecture. Consider the following functors with values in the additive category of $\bbZ[1/s]$-modules:
\begin{eqnarray}\label{eq:functors-3}
&& E_m(-)_{1/s}\colon \dgcat(k) \too \mathrm{Mod}(\bbZ[1/s]) \quad \cA \mapsto (\pi_{-1}L_{KU} K(\cA\otimes_{\bbF_q} \bbF_{q^m}))_{1/s}\,.
\end{eqnarray}
Similarly to the proof of Proposition \ref{prop:additive}, these functors are additive invariants. Consequently, using the fact that $1/s \in k$ (since $p\nmid s$) and that  category $\mathrm{Mod}(\bbZ[1/s])$ is $\bbZ[1/s]$-linear, we conclude from \cite[Thm.~1.1 and Rk.~1.4]{Orbifold} that $E_m(\perf_\dg(\cX))_{1/s}$ is a direct summand of $\bigoplus_{\sigma\subseteq G} E_m(\perf_\dg(X^\sigma \times \mathrm{Spec}(k[\sigma])))_{1/s}$. By applying the functor $\Hom(\bbZ(l^\infty),-)$, we hence obtain the following implication:
\begin{equation}\label{eq:implication-aux1}
\sum_{\sigma \subseteq G} T^l(X^\sigma \times \mathrm{Spec}(k[\sigma]); 1/s) \Rightarrow T^l(\cX; 1/s)\,.
\end{equation}
The searched implication \eqref{eq:implication-main1} follows then from the combination of \eqref{eq:implication-aux1} with Lemma \ref{lem:aux}. Assume now that $s\mid (q-1)$. Note that $s\mid (q-1)$ if and only if $k$ contains the $s^{\mathrm{th}}$ roots of unity. Therefore, since $1/s \in k$, since the above functors \eqref{eq:functors-3} are additive invariants, and since the category $\mathrm{Mod}(\bbZ[1/s])$ is $\bbZ[1/s]$-linear, we conclude from \cite[Cor.~1.6(i)]{Orbifold} that $E_m(\perf_\dg(\cX))_{1/s}$ is a direct summand of $\bigoplus_{\sigma \subseteq G} E_m(\perf_\dg(X^\sigma))^{\oplus r_\sigma}_{1/s}$, where $r_\sigma \geq 1$ are certain integers. By applying the functor $\Hom(\bbZ(l^\infty),-)$, we hence obtain the following implication:
\begin{equation}\label{eq:implication-aux}
\sum_{\sigma \subseteq G} T^l(X^\sigma; 1/s) \Rightarrow T^l(\cX; 1/s)\,.
\end{equation}
As above, the searched implication $\sum_{\sigma \subseteq G} T^l(X^\sigma) \Rightarrow T^l(\cX)$ (for every $l\nmid s$) follows then from the combination of \eqref{eq:implication-aux} with Lemma \ref{lem:aux}. Assume finally that $\mathrm{dim}(X)\leq 3$. In this case, the dimension of the smooth projective $k$-schemes $X^\sigma$ and $X^\sigma \times \mathrm{Spec}(k[\sigma])$ is also $\leq 3$. Therefore, as explained in \S\ref{sub:divisors}, it suffices to consider the Tate conjecture for divisors $T^{l,1}(-)$. As proved in \cite[Thm.~5.2]{Tate-motives}, we have the following equivalences:
$$ T^{l,1}(X^\sigma \times \mathrm{Spec}(k[\sigma]))\Leftrightarrow T^{l,1}(X^\sigma) + T^{l,1}(\mathrm{Spec}(k[\sigma]))\Leftrightarrow T^{l,1}(X^\sigma)\,.$$
This yields the equivalence of conjectures $T^l(X\times \mathrm{Spec}(k[\sigma])) \Leftrightarrow T^l(X^\sigma)$ and hence the searched implication $\sum_{\sigma \subseteq G} T^l(X^\sigma) \Rightarrow T^l(\cX)$ (for every $l\nmid s$). 

Let us now prove Theorem \ref{thm:orbifold} in what regards the $p$-version of the Tate conjecture. Since the functor \eqref{eq:universal} is additive, $1/s \in k$, and $\mathrm{char}(K)=0$, \cite[Thm.~1.1 and Rk.~1.4]{Orbifold} implies that $U(\perf_\dg(\cX))_K$ is a direct summand of the direct sum $\bigoplus_{\sigma\subseteq G} U(\perf_\dg(X^\sigma \times \mathrm{Spec}(k[\sigma])))_K$. Recall from the proof of Lemma \ref{lem:factorization} that the functor \eqref{eq:TP} comes equipped with the natural transformation $\varphi^n$. Therefore, since the dg categories $\perf_\dg(\cX)$ and $\perf_\dg(X^\sigma \times \mathrm{Spec}(k[\sigma]))$ are smooth proper, by applying the functor \eqref{eq:TP} we conclude that the induced $K$-linear homomorphism 
$$ K_0(\perf_\dg(\cX))_K \too TP_0(\perf_\dg(\cX))_{1/p}^{\varphi^n}$$
is a direct summand of the induced (diagonal) $K$-linear homomorphism
$$\bigoplus_{\sigma \subseteq G} K_0(X^\sigma \times \mathrm{Spec}(k[\sigma]))_K \too  \bigoplus_{\sigma \subseteq G} TP_0(X^\sigma \times \mathrm{Spec}(k[\sigma]))_{1/p}^{\varphi^n}\,.$$
This yields the searched implication \eqref{eq:implication-main2} Assume now that $s\mid (q-1)$. In this case, since $k$ contains the $s^{\mathrm{th}}$ roots of unity, \cite[Cor.~1.6(i)]{Orbifold} implies that $U(\perf_\dg(\cX))_K$ is a direct summand of $\bigoplus_{\sigma\subseteq G} U(\perf_\dg(X^\sigma))^{r_\sigma}_K$, where $r_\sigma$ are certain integers. Hence, an argument similar to the preceding one shows that $\sum_{\sigma \subseteq G} T^p(X^\sigma)~\Rightarrow~T^p(\cX)$. Assume finally that $\mathrm{dim}(X)\leq 3$. As explained in \S\ref{sub:divisors}, since the dimension of $X^\sigma$ and $X^\sigma \times \mathrm{Spec}(k[\sigma])$ is $\leq 3$, we have $T^p(X^\sigma) \Leftrightarrow T^l(X^\sigma)$ and $T^p(X^\sigma \times \mathrm{Spec}(k[\sigma]))\Leftrightarrow T^l(X^\sigma \times \mathrm{Spec}(k[\sigma]))$. Consequently, making use of the above equivalence of conjectures $T^l(X\times \mathrm{Spec}(k[\sigma])) \Leftrightarrow T^l(X^\sigma)$, we conclude that $\sum_{\sigma \subseteq G} T^p(X^\sigma) \Rightarrow T^p(\cX)$.

Let us now prove Theorem \ref{thm:orbifold} in what regards the Beilinson conjecture. As above, $U(\perf_\dg(\cX))_\bbQ$ is a direct summand of $\bigoplus_{\sigma\subseteq G} U(\perf_\dg(X^\sigma \times \mathrm{Spec}(k[\sigma])))_\bbQ$. Therefore, thanks to the natural isomorphisms \eqref{eq:isos} and \eqref{eq:numerical}, by applying the functors $\Hom_{\NChow(k)_\bbQ}(U(k)_\bbQ, -)$ and $\Hom_{\NNum(k)_\bbQ}(U(k)_\bbQ, -)$, we obtain the implication \eqref{eq:implication-main3}. Once again as above, whenever $m\mid (q-1)$, $U(\perf_\dg(\cX))_\bbQ$ is a direct summand of $\bigoplus_{\sigma \subseteq G} U(\perf_\dg(X^\sigma))^{\oplus r_\sigma}_\bbQ$. Hence, an argument similar to the preceding one shows that $\sum_{\sigma \subseteq G} B(X^\sigma) \Rightarrow B(\cX)$.

Finally, since the functors \eqref{eq:K-theory} are additive invariants (with values in the additive category of $\bbQ$-vector spaces), the proof of Theorem \ref{thm:orbifold} in what regards the Parshin conjecture is similar to the above proof of Theorem \ref{thm:orbifold} in what regards the Tate conjecture.
\section*{Proof of Theorem \ref{prop:twisted}}
We start by proving Theorem \ref{prop:twisted} in what regards the Tate conjecture. Since by assumption $s\mid (q-1)$, $k$ contains the $s^{\mathrm{th}}$ roots of unity and $1/s \in k$. Therefore, using the fact that the above functors \eqref{eq:functors-3} (with $s$ replaced by $sr$) are additive invariants and that the category $\mathrm{Mod}(\bbZ[1/sr])$ is $\bbZ[1/sr]$-linear, we conclude from \cite[Cor.~1.29(ii)]{Orbifold} that $E_m(\perf_\dg(\cX; \cF))_{1/sr}$ is a direct summand of $\bigoplus_{\sigma \subseteq G} E_m(\perf_\dg(Y_\sigma))_{1/sr}$, where $Y_\sigma$ is a certain $\sigma^\vee$-Galois cover of $X^\sigma$. By applying the functor $\Hom(\bbZ(l^\infty),-)$, we hence obtain the following implication: 
\begin{equation*}
\sum_{\sigma \subseteq G} T^l(Y_\sigma; 1/sr) \Rightarrow T^l(\cX;\cF; 1/sr)\,.
\end{equation*}
Lemma \ref{lem:aux} yields then the following implication:
\begin{eqnarray}\label{eq:implication-big1}
 \sum_{\sigma \subseteq G} T^l(Y_\sigma) \Rightarrow T^l(\cX;\cF) && (\mathrm{for}\,\mathrm{every}\,\,l \nmid sr)\,.
\end{eqnarray}
On the one hand, if $\mathrm{dim}(X)\leq 1$, then the dimension of $Y_\sigma$ is also $\leq 1$ for every $\sigma \subseteq G$. Consequently, the conjecture $T^l(\cX;\cF)$ follows from the implication \eqref{eq:implication-big1}. On the other hand, if the $G$-action is faithful and $\mathrm{dim}(X)=2$, then $\mathrm{dim}(Y_\sigma)=\mathrm{dim}(X^\sigma)\leq 1$ for every non-trivial cyclic subgroup $\sigma \subseteq G$. Consequently, in the above implication \eqref{eq:implication-big1} the $k$-schemes $Y_\sigma$ can be replaced by the single~$k$-scheme~$X$.

Let us now prove Theorem \ref{prop:twisted} in what regards the $p$-version of the Tate conjecture. Since the functor \eqref{eq:universal} is additive, $k$ contains the $s^{\mathrm{th}}$ roots of unity, and $\mathrm{char}(K)=0$, \cite[Cor.~1.29(ii)]{Orbifold} implies that $U(\perf_\dg(\cX;\cF))_K$ is a direct summand of $\bigoplus_{\sigma \subseteq G} U(\perf_\dg(Y_\sigma))_K$. Recall from the proof of Lemma \ref{lem:factorization} that the functor \eqref{eq:TP} comes equipped with the natural transformation $\varphi^n$. Therefore, since the dg categories $\perf_\dg(\cX;\cF)$ and $\perf_\dg(Y_\sigma)$ are smooth proper, by applying the functor \eqref{eq:TP} we conclude that the induced $K$-linear homomorphism 
$$ K_0(\perf_\dg(\cX;\cF))_K \too TP_0(\perf_\dg(\cX;\cF))_{1/p}^{\varphi^n}$$
is a direct summand of the induced (diagonal) $K$-linear homomorphism
$$ \bigoplus_{\sigma \subseteq G} K_0(\perf_\dg(Y_\sigma))_K \too \bigoplus_{\sigma \subseteq G} TP_0(\perf_\dg(Y_\sigma))_{1/p}^{\varphi^n}\,.$$
This yields the implication $
\sum_{\sigma \subseteq G} T^p(Y_\sigma) \Rightarrow T^p(\cX;\cF)$. The remainder of the proof is now similar to the above proof, concerning the Tate conjecture, with the implication \eqref{eq:implication-big1} replaced by the latter implication.

Let us now prove Theorem \ref{prop:twisted} in what regards the Beilinson conjecture. As above, $U(\perf_\dg(\cX; \cF))_\bbQ$ is a direct summand of $\bigoplus_{\sigma \subseteq G} U(\perf_\dg(Y_\sigma))_\bbQ$. Therefore, thanks to the natural isomorphisms \eqref{eq:isos} and \eqref{eq:numerical}, by applying the functors $\Hom_{\NChow(k)_\bbQ}(U(k)_\bbQ, -)$ and $\Hom_{\NNum(k)_\bbQ}(U(k)_\bbQ, -)$, we obtain the implication $\sum_{\sigma \subseteq G} B(Y_\sigma) \Rightarrow B(\cX;\cF)$. The remainder of the proof is now similar to the above proof, concerning the Tate conjecture, with \eqref{eq:implication-big1} replaced by the latter implication.

Finally, since the functors \eqref{eq:K-theory} are additive invariants (with values in the additive category of $\bbQ$-vector spaces), the proof of Theorem \ref{prop:twisted} in what regards the Parshin conjecture is similar to the above proof of Theorem \ref{prop:twisted} in what regards the Tate conjecture.

\medbreak\noindent\textbf{Acknowledgments:} I am grateful to Bruno Kahn for his interest in the arguments used in the proof of Theorem \ref{thm:two} and for the reference \cite{Reid}, to Kiran Kedlaya for asking me if the $p$-version of the Tate conjecture would admit a noncommutative analogue, and to Lars Hesselholt and Peter Scholze for useful discussions concerning topological periodic cyclic homology. I also would like to thank the Hausdorff Research Institute for Mathematics in Bonn for its hospitality.

\end{document}

\end{proof}